\documentclass[a4paper,11pt,reqno]{amsart}
\usepackage{graphicx}
\usepackage{enumitem}
\usepackage{mathtools}
\usepackage{hyperref}

\usepackage{pstricks,tikz}
\usetikzlibrary{positioning,chains,fit,shapes,calc,fit}

\newtheorem{theorem}{Theorem}[section]
\newtheorem{lemma}[theorem]{Lemma}
\newtheorem{proposition}[theorem]{Proposition}
\newtheorem{corollary}[theorem]{Corollary}
\newtheorem{conjecture}[theorem]{Conjecture}
\newtheorem{claim}[theorem]{Claim}

\newtheorem{definition}[theorem]{Definition}

\numberwithin{equation}{section}

\def\eps{{\varepsilon}}

\def\int{{\mbox{\rm int}}}

\def\disc{{\mbox{\rm ex}}}
\def\ex{{\mbox{\rm ex}}}
\def\pn{{\mbox{\rm pn}}}

\def\COMMENT#1{}

\title{Decomposing tournaments into paths}

\author{Allan Lo}
\address{Allan Lo, School of Mathematics, University of Birmingham, Edgbaston, Birmingham, B15 2TT, UK.}
\email{s.a.lo@bham.ac.uk}
\thanks{A. Lo was partially supported by the EPSRC grant no. EP/P002420/1 (A.~Lo).}

\author{Viresh Patel}
\address{Viresh Patel, Korteweg de Vries Institute for Mathematics, University of Amsterdam, Amsterdam, The Netherlands.}
\email{vpatel@uva.nl}
\thanks{V. Patel was partially supported by the Netherlands Organisation for Scientific Research (NWO) through the Gravitation Programme Networks (024.002.003).}

\author{Jozef Skokan}
\address{Jozef Skokan, Department of Mathematics, London School of Economics, London, UK \and  \\
University of Illinois at Urbana-Champaign, Urbana, IL 61801, USA.} 
\email{j.skokan@lse.ac.uk}
\thanks{J. Skokan was partially supported by National Science Foundation Grant DMS-1500121.}

\author{John Talbot}
\address{John Talbot, Department of Mathematics, University College London, London, WC1E 6BT, UK.}
\email{j.talbot@ucl.ac.uk}

\date{\today}

\begin{document}

\maketitle
\begin{abstract}
We consider a generalisation of Kelly's conjecture which is due to Alspach, Mason, and Pullman from 1976. Kelly's conjecture states that every regular tournament has an edge decomposition into  Hamilton cycles, and this was proved by K{\"u}hn and Osthus for large tournaments. The conjecture of Alspach, Mason, and Pullman 
asks for the minimum number of paths needed in a path decomposition  of a general tournament $T$. There is a natural lower bound for this number in terms of the degree sequence of $T$ and it is  conjectured that this bound is correct for tournaments of even order. Almost all cases of the conjecture are open and we prove many of them.
\end{abstract}

\section{Introduction}
\label{se:intro}
\thispagestyle{empty}


There has been a great deal of recent activity in the study of decompositions of graphs and hypergraphs. The  prototypical question in this area asks whether, for some given class $\mathcal{C}$ of graphs, hypergraphs or directed graphs, the edge set of each $H \in \mathcal{C}$ can be decomposed into parts satisfying some given property. 
The development of the robust expanders technique by K{\"u}hn and Osthus \cite{KO} was a major breakthrough leading to the resolution of several conjectures concerning decompositions of (directed) graphs into spanning structures such as matchings and Hamilton cycles; see e.g.\ \cite{CKLOT,KO2}.

The problem we address in this paper is that of decomposing tournaments into directed paths. A \emph{tournament} is an orientation of the complete graph, that is, one obtains a tournament by assigning a direction to each edge of the (undirected) complete graph. Let us begin however in the more general setting of directed graphs. 

Let $D$ be a directed graph with vertex set $V(D)$ and edge set $E(D)$.
When referring to paths and cycles in directed graphs, we always mean directed paths and directed cycles.
  A \emph{path decomposition} of $D$ is a collection of paths $P_1, \ldots, P_k$ of $D$ whose edge sets $E(P_1), \ldots, E(P_k)$ partition $E(D)$. Given any directed graph $D$, it is natural to ask what the minimum number of paths is in a path decomposition of $D$. 
This is called the \emph{path number} of $D$ and is denoted $\pn(D)$. A natural lower bound on $\pn(D)$ is obtained by examining the degree sequence of $D$. For each vertex $v \in V(D)$, write $d^+_D(v)$ (resp.\ $d_D^-(v)$) for the 
number of edges exiting (resp.\ entering) $v$. 
Define the \emph{excess} at $v$ to be $\ex(v):= d_D^+(v) - d_D^-(v)$ and similarly define the \emph{positive} and \emph{negative excess} at $v$ to be respectively $\ex^+(v):= \max\{\ex(v), 0\}$ and $\ex^-(v):= \max\{-\ex(v), 0\}$. It is easy to see that the excesses of all vertices sum to zero. 

We note that in any path decomposition of $D$, at least $\ex^+(v)$ paths must start at $v$ and at least $\ex^-(v)$ paths must end at $v$. Therefore we have
\[
\pn(D) \geq \ex(D) := \sum_{v \in V(D)} \ex^+(v) = \sum_{v \in V(D)} \ex^-(v) = \frac{1}{2}\sum_{v \in V(D)} |\ex(v)|  ,
\]
where $\ex(D)$ is called the \emph{excess} of $D$. Any digraph for which equality holds above is called \emph{consistent}.
Clearly not every digraph is consistent; in particular any nonempty digraph $D$ of excess $0$ cannot be consistent. However, Alspach, Mason, and Pullman~\cite{AMP} conjectured that every even tournament is consistent.

\begin{conjecture}[Alspach, Mason, Pullman~\cite{AMP}]
\label{conj:Pull}
Every tournament $T$ with an even number of vertices satisfies $\pn(T) = \ex(T)$.
\end{conjecture}

It is almost immediate to see that this conjecture is a considerable generalisation of Kelly's conjecture stated below. We give the easy argument after Theorem~\ref{thm:KO}.

\begin{conjecture} [Kelly; see e.g.\ \cite{BG}]
\label{conj:Kelly}
The edge set of every regular tournament can be decomposed into Hamilton cycles.
\end{conjecture}

Almost 50 years after it was stated, K{\"u}hn and Osthus~\cite{KO} finally proved Kelly's conjecture for large tournaments using their powerful robust expanders technique, which was subsequently used to prove several other conjectures on edge decompositions of (directed) graphs \cite{KO2,CKLOT}.

\begin{theorem}[K{\"u}hn, Osthus~\cite{KO}]
\label{thm:KO}
Every sufficiently large regular tournament has a Hamilton decomposition.
\end{theorem}

To see that Conjecture~\ref{conj:Pull} implies Conjecture~\ref{conj:Kelly}, take any regular $(n+1)$-vertex tournament $T$ and any $v \in V(T)$, and note that $\ex(T-v) = n/2$. If Conjecture~\ref{conj:Pull} holds, then $T-v$ can be decomposed into $n/2$ paths, so they must be Hamilton paths. Adding $v$ back to $T-v$, it is easy to see that each path can be completed to a Hamilton cycle, giving a Hamilton decomposition of $T$. The converse is also easy to see. Thus the special case of Conjecture~\ref{conj:Pull} in which  $\ex(T) = n/2$ is equivalent to Kelly's Conjecture. In general, however, $\ex(T)$ can take a large range of values as the proposition below shows.

\begin{proposition}
If $T$ is an $n$-vertex tournament with $n$ even, then $n/2 \leq \ex(T) \leq n^2/4$. Furthermore each value in the range occurs.
\end{proposition}
As we saw, the lower bound occurs for any almost-regular tournament and it is easy to verify that the upper bound occurs for the transitive tournament (in fact it occurs for any tournament with a vertex partition into two equal size parts $A$ and $B$ where all edges are directed from $A$ to $B$). Alspach and Pullman~\cite{AP} showed that for any tournament $T$, $\pn(T) \leq n^2/4$ thus verifying Conjecture~\ref{conj:Pull} for the special case $ex(T) = n^2/4$  (and this was generalised to digraphs \cite{Brien}). Thus the conjecture has been solved for the two extreme values of excess, namely $n/2$ and $n^2/4$: for every other value of $\ex(T)$ between $n/2$ and $n^2/4$ the conjecture remains open. Our main contribution is to solve many more cases of the conjecture.

\begin{theorem}
\label{thm:LPST2}
There exists $C>0$ and $n_0 \in \mathbb{N}$ such that if $T$ is an $n$-vertex tournament with $n \geq n_0$ even and $\ex(T) > Cn$ then $\pn(T) = \ex(T)$.
\end{theorem}
We make no attempt to optimise or even compute the value of $C$ but we note it is not a Regularity-type constant. 
We prove this theorem in two steps. We will first prove the  following weakening of the Theorem~\ref{thm:LPST2}.
\begin{theorem}
\label{thm:LPST1}
There exists $\eps > 0$ (we can take $\eps = 1/18$) and $n_0 \in \mathbb{N}$ such that
if $T$ is a tournament on $n>n_0$ vertices with $n$ even and $\ex(T) \geq n^{2 - \eps}$, then $\pn(T) = \ex(T)$.  
\end{theorem}
The proof of this result is short and self-contained, relying on a novel application of the absorption technique due to R{\"o}dl, Ruci{\'n}ski, and Szemer{\'e}di~\cite{RRS} (with special forms appearing in earlier work e.g.\ \cite{Kriv}). 

In the next step we consider tournaments of excess smaller than $n^{2 - \eps}$ but bigger than $Cn$. Such tournaments are almost regular and are therefore amenable to the techniques used by K{\"u}hn and Osthus~\cite{KO}. For tournaments of small excess, we will ultimately reduce the problem of showing that $\pn(T) = \ex(T)$ to the problem of showing that a regular oriented graph $D$ of very high degree has an edge decomposition into Hamilton cycles; such a Hamilton decomposition of $D$ is known to exist by the main result  from~\cite{KO}. 

\subsection{Outline}
In the next section, we give the basic notation we will use as well as some preliminary results needed in Section~\ref{se:highdisc}. In Section~\ref{se:highdisc} we give the short proof of Theorem~\ref{thm:LPST1}, which requires only Hall's Theorem and Menger's Theorem. 
In Section~\ref{se:fprelims} we give further preliminaries needed for the remaining sections; in particular we state the results related to robust expansion that we will need. At the end of Section~\ref{se:fprelims} we give an overview of the arguments in Section~\ref{se:W} and Section~\ref{se:finaldecomp} that allow us to extend Theorem~\ref{thm:LPST1} to Theorem~\ref{thm:LPST2}.
Section~\ref{se:W} contains a preliminary result, Lemma~\ref{lma:W}, that helps us to deal with certain problematic vertices that we encounter in Section~\ref{se:finaldecomp}.
In Section~\ref{se:finaldecomp}, we prove Theorem~\ref{thm:LPST2} in a three-step reduction via Theorem~\ref{thm:RemHighDisc}, Theorem~\ref{thm:DiscBalance}, and Theorem~\ref{thm:FinalDecomp}.

\section{Notation and preliminaries}
\label{se:prelim}

\subsection{Notation}

In this paper a digraph refers to a directed graph without loops where we allow up to two edges between any pair $x$, $y$ of distinct vertices, at most one in each direction. Occasionally we work with directed multigraphs which again have no loops, but where we permit more than two directed edges between any pair of distinct vertices. An oriented graph is a directed graph where we permit only one edge between any pair of distinct vertices.
  Given a digraph $D$, we write $V(D)$ for its vertex set and $E(D)$ for its edge set. We write $xy$ for an edge directed from $x$ to $y$.

We write $H \subseteq D$ to mean that $H$ is a subdigraph of $D$, i.e.\ $V(H) \subseteq V(D)$ and $E(H) \subseteq E(D)$.
Given $X\subseteq V(D)$, we write $D-X$ for the digraph obtained from $D$ by deleting all vertices in~$X$,
and $D[X]$ for the subdigraph of $D$ induced by~$X$.
Given $F\subseteq E(D)$, we write $D- F$ for the digraph obtained from $D$ by deleting all edges in~$F$. If $H$ is a subdigraph of $D$, we write $D- H$ for $D- E(H)$. For two subdigraphs $H_1$ and $H_2$ of $D$, we write $H_1 \cup H_2$ for the subdigraph with vertex set $V(H_1) \cup V(H_2)$ and edge set $E(H_1) \cup E(H_2)$. For a set of edges $F \subseteq E(D)$, we sometimes write $V(F)$ to denote the set of vertices incident to some edge in $F$.

If $x$ is a vertex of a digraph $D$, then $N^+_D(x)$ denotes the \emph{out-neighbourhood} of $x$, i.e.~the
set of all those vertices $y$ for which $xy\in E(D)$. Similarly, $N^-_D(x)$ denotes the \emph{in-neighbourhood} of $x$, i.e.~the
set of all those vertices $y$ for which $yx\in E(D)$.
For $S \subseteq V(D)$, we write $N^+_D(x, S)$ for all those vertices $y \in S$ such that $xy \in E(D)$ and correspondingly for $N^-_D(x,S)$. 
We write $d^+_D(x):=|N^+_D(x)|$ for the \emph{outdegree} of $x$ and $d^-_D(x):=|N^-_D(x)|$ for its \emph{indegree}. 
Similarly we write $d^{\pm}_D(x, S):=|N^{\pm}_D(x, S)|$.
We denote the \emph{minimum outdegree} of $D$ by $\delta^+(D):=\min \{d^+_D(x): x\in V(D)\}$
and the \emph{minimum indegree} $\delta^-(D):=\min \{d^-_D(x): x\in V(D)\}$.
The \emph{minimum semi-degree} of $D$ is $\delta^0(D):=\min\{\delta^+(D), \delta^-(D)\}$ and the minimum degree is $\delta(D) := \min\{d^+(x)+ d^-(x): x \in V(D)\}$. We use $\Delta^{\pm}(D)$, $\Delta^0(D)$ and $\Delta(D)$ for the corresponding maximum degrees.

Whenever $X,Y\subseteq V(D)$ are disjoint, we write $E_D(X)$ for the set of edges of $D$ having both endvertices in~$X$, and $E_D(X,Y)$ for the set of edges of $D$ that start in $X$ and end in $Y$. 

Unless stated otherwise, when we refer to paths and cycles in digraphs, we mean
directed paths and cycles, i.e.~the edges on these paths and cycles are oriented consistently. We write $P = x_1x_2 \cdots x_t$ to indicate that $P$ is a path with edges $x_1x_2, x_2x_3, \ldots, x_{t-1}x_t$, where $x_1, \ldots, x_t$ are distinct vertices. We occasionally denote such a path $P$ by $x_1Px_t$ to indicate that it starts at $x_1$ and ends at $x_t$. 
For two paths $P=a \cdots b$ and $Q = b \cdots c$, we write $aPbQc$ for the concatenation of the paths $P$ and $Q$ and this notation generalises to cycles in the obvious ways. In particular for a cycle $C$ and vertices $a,b$ on the cycle, $aCb$ denotes the paths from $a$ to $b$ along the cycle.
We often use calligraphic letters, e.g.\ $\mathcal{P}$ for a set of paths $\mathcal{P} = \{P_1, \ldots, P_r \}$. In that case $\cup \mathcal{P}$ refers to the digraph that is the union of the paths and $V(\mathcal{P})$ and $E(\mathcal{P})$ refer to the vertex and edge set of the union.

  For a set $X$ and $U \subseteq X$, we will write $I_U: X \rightarrow \{0,1\}$ for the indicator function of~$U$.

For $x, y \in (0,1]$, we often use the notation $x \ll y$ to mean that $x$ is sufficiently small as a function of $y$ i.e.\ $x \leq f(y)$ for some implicitly given non-decreasing function $f:(0,1] \rightarrow (0,1]$. 

Throughout, we omit floors and ceilings and treat large numbers as integers whenever this does not affect the argument.

\subsection{Basic graph theory}

We will very occasionally work with undirected graphs for which we use standard notation similar to that used for directed graphs; see e.g.\ \cite{Diestel}.

\begin{theorem}[variant of Hall's Theorem]
\label{thm:Hall}
Suppose $G$ is a bipartite graph with vertex classes $A$ and $B$ and $k \in \mathbb{N}$. If $k|N_G(X)| \geq |X|$ for every $X \subseteq A$, then each $a \in A$ can be matched with some $b \in B$ such that each $b \in B$ is matched with at most $k$ elements of $A$, i.e.\ there exists a subgraph $G' \subseteq G$ in which every vertex in $A$ has degree $1$ and every vertex in $B$ has degree at most $k$.
\end{theorem}

\begin{corollary}
\label{cor:highdegmatch}
Suppose $G$ is a bipartite graph with vertex classes $A$ and $B$ both of size $n$ and suppose $\delta(G) \geq n/2$. Then $G$ has a perfect matching.
\end{corollary}

For a directed graph $D$ and $A, B \subseteq V(D)$, an \emph{$A,B$-path} of $D$ is a path of $D$ that starts in $A$ and ends in $B$. An \emph{$A,B$-separator} of $D$ is a vertex subset $S \subseteq V(D)$ such that there are no $A,B$-paths in $D - S$.
\begin{theorem}[Menger's Theorem]
Suppose $D$ is a directed graph and $A, B \subseteq V(D)$. If the smallest $A,B$-separator in $D$ has size $t$, then there exist $t$ internally vertex-disjoint $A, B$-paths in $D$.
\end{theorem}

\subsection{Excess and partial decompositions}
We recall definitions from the introduction. Let $D$ be a directed graph. 
For a vertex $v \in V(D)$, recall that $\disc_D(v) := d^+_D(v) - d^-_D(v)$.
We define $\disc^+_D (v) : = \max \{0, \disc_D(v)\}$ and $\disc^-_D (v) : = \max \{0, -\disc_D(v)\}$.
Let 
\[
\disc(D) : = \frac{1}{2}\sum_{v \in V(D)} | \disc_D(v) | = \sum_{v \in V(D)}  \disc^+_D(v) = \sum_{v \in V(D)}  \disc^-_D(v).
\]
For $\ast \in \{+,-\}$, let $U^*(D) : = \{v \in V(D) : \disc_D^*(v) > 0\}$ and let $U^0(D): = \{v \in V(D) : \disc_D(v) = 0\}$. 

We state the following very simple observation  so we can refer to it later.
\begin{proposition}
\label{pr:disc}
Suppose $D$ is a directed graph and $H \subseteq D$ is a subdigraph in which $\disc^*_H(v) \leq \disc^*_D(v)$ for all $v \in V(D)$ and $\ast \in \{+,-\}$. Then $\disc(D) = \disc(H) + \disc(D-H)$
\end{proposition}
\begin{proof}
To see this, note that either $\disc_D(v)$, $\disc_H(v)$, and $\disc_{D-H}(v)$ are all at least zero or all at most zero for each $v \in V(D)$. Hence $\disc_D(v) = \disc_H(v) + \disc_{D-H}(v)$ for all $v \in V(D)$. We sum over all vertices to obtain the result.
\end{proof}

The following definitions are convenient.
\begin{definition}
A \emph{perfect decomposition} of a digraph $D$ is a set  $\mathcal{P} = \{P_1, \ldots, P_r \}$ of edge-disjoint paths of $D$ that together cover $V(D)$ with  $r = \disc(D)$. 
(Thus Conjecture~\ref{conj:Pull} states that every even tournament has a perfect decomposition.)

A \emph{partial decomposition} of a digraph $D$ is a set   $\mathcal{P} = \{ P_1, \ldots, P_k \}$ of edge-disjoint paths of $D$ such that
for every $v \in V(D)$ at most $\disc^+_D(v)$ of the  paths start at $v$ and at most $\disc^-_D(v)$ of the paths end at $v$.
\end{definition}

It is easy to see that any subset of a perfect decomposition of $G$ is a partial decomposition of $G$.
We will need the following straightforward fact about perfect decompositions.

\begin{proposition}
\label{pr:acyclic}
If $D$ is an acyclic digraph then it has a perfect decomposition.
\end{proposition}
\begin{proof}
Iteratively remove paths of maximum length. Note that removing such a path from an acyclic digraph reduces the excess by one (since such a path must begin at a vertex $v$ where $d^-(v)=0$ (and hence $\disc(v)>0$), and must end at a vertex where $d^+(v)=0$ (and hence $\disc(v)<0$). So the proposition holds by induction. 
\end{proof}

\section{Exact Decomposition for tournaments with high excess}
\label{se:highdisc}


In this section we prove Theorem~\ref{thm:LPST2}.
We start by showing that any Eulerian oriented graph can be decomposed into a small number of cycles. We will also need an extra technical condition on this cycle decomposition.
We use the following result of Huang, Ma, Shapira, Sudakov, Yuster~\cite[Proposition 1.5]{HMSSY}.

\begin{lemma}
\label{lma:sudakov(new)}
Every Eulerian digraph $D$ with $n$ vertices and $m$ edges has a cycle of length $1 + \max(m^2/24n^3 , \lfloor \sqrt{m/n} \rfloor )$.
\end{lemma}

\begin{lemma} \label{lma:EulerDecomp(new)}
Let $n \in \mathbb{N}$.
 Let $D$ be an Eulerian oriented graph with $n$ vertices. Then we can decompose $D$ into $t \le 50 n^{4/3} \log n$ cycles $C_1, \ldots, C_t$ and for each cycle $C_i$ we can find distinct representatives $x^i_1, x^i_2, \ldots, x^i_{r_i} \in V(C_i)$ (indexed in order) with the following properties:
 \begin{enumerate}[label = {\rm (\roman*)}]
	\item Every cycle has at least two representatives, i.e. $r_i \geq 2$ for all $i$;
	\item The interval between consecutive vertices on a cycle $x^i_{j}C_i x^i_{j+1}$ has length at most $n^{2/3}$;
	\item Every vertex $v \in V$ occurs as a representative at most $24 n^{2/3}\log^{1/2}n$ times.
\end{enumerate} 
\end{lemma}

\begin{proof}
We first show that $D$ can be decomposed into at most $50n^{4/3} \log n$ cycles. Assume $D$ has $m < n^2/2$ edges.
We iteratively remove the longest cycle and let $m_t$ be the number of edges remaining at step $t$. From Lemma~\ref{lma:sudakov(new)} we have that $m_{t+1} \leq m_t - g(m_t)$ where 
\[
g(r) = \max\{r^2/24n^3 , \lfloor \sqrt{r/n} \rfloor \} > \frac{r}{24n^{4/3}}.
\] 
To see the inequality note that if $r  \ge n^{5/3}$, then $r^2/24n^3 \ge r/24n^{4/3}$, and if $r <n^{5/3}$, then $\sqrt{r/n} > r/n^{1/2+5/6} = r/n^{4/3}$.
Thus we see that 
\[
m_{t+1} < m_t -  \frac{m_t}{24n^{4/3}} = m_t\left( 1 - \frac{1}{24n^{4/3}} \right) \leq m_t \exp\left( - \frac{1}{24n^{4/3}} \right).
\]
Hence $m_t < \exp(-t/24n^{4/3})n^2$ from which we see that $m_t < 1$ after at most $50n^{4/3} \log n$ steps, giving at most as many cycles in the greedy decomposition of $D$.

Next we show how to obtain the representatives. Assume we have a decomposition of $D$ into a minimum number of cycles $C_1, \ldots, C_t$, where we know $t \leq 50 n^{4/3} \log n$.

 First we treat the long cycles. Assume without loss of generality that $C_1, \ldots, C_k$ are the cycles in our decomposition of length larger than $n^{2/3}$. Divide each such cycle $C_i$ into intervals $I^i_1, \dots, I^i_{r_i}$ each of length between $n^{2/3}/4$ and $n^{2/3}/2$ with $r_i$ minimal. Note that $ r_i \le 4 |E(C_i)| n^{-2/3} $ for all $i \in [k]$. Thus in total we have at most $\sum_{i \in [k]} 4 |E(C_i)| n^{-2/3} \le 4|E(D)|n^{-2/3} \le 2 n^{4/3}$ intervals each of length at least $n^{2/3}/4$.
Therefore, we can greedily pick $x^i_j \in I^i_j$ such that no vertex in $V(D)$ appears as a representative more than $8 n^{2/3}$ times.   

Consider the remaining (short) cycles $C_{k+1}, \ldots, C_t$, for which we need only find two representatives each. 
Let $\mathcal{C} = \{C_{k+1}, \ldots C_t \}$.
First, we will find one representative in each cycle of $\mathcal{C}$ such that no vertex is chosen more than $8 n^{2/3}\log^{1/2}n$ times.
Let $H$ be the bipartite graph with vertex partitions $\mathcal{C}$ and $V(D)$, where for $C \in \mathcal{C}$ and $v \in V(D)$ are joined if and only if $v \in V(C)$.
We now apply (a version of) Hall's theorem (Theorem~\ref{thm:Hall}) to find one representative in each $\mathcal{C}$ such that no vertex is chosen more than $8 n^{2/3}\log^{1/2}n$ times.
If such a collection of representatives does not exist, then Theorem~\ref{thm:Hall} implies that there exists a subset $\mathcal{C}'$ of~$\mathcal{C}$ such that $8 n^{2/3}\log^{1/2}n |N_H (\mathcal{C}') |  < |\mathcal{C}'|  $.
On the other hand, we have
\begin{align*}
	|N_H (\mathcal{C}') | = | V ( \bigcup \mathcal{C}' ) | \ge |E(\bigcup \mathcal{C}')|^{1/2} \ge  | \mathcal{C}'| ^{1/2}.
\end{align*}
This implies that $t \ge |\mathcal C'| > 64 n^{4/3} \log n$, a contradiction.
Thus we have found one representative $x^i_1 \in V(C_i)$ for each $k+1 \le i \le t$ such that each vertex $v \in V$ occurs as a representative at most $8 n^{2/3}\log^{1/2}n$ times.
Next let $P_i := C_i \setminus x^i_1$ for each $k+1 \le i \le t$.
Note that $|E(P_i)| \ge 1$. 
By a similar argument as above, we can find one representative $x^i_2 \in V(P_i)$ for each $k+1 \le i \le t$ such that each vertex $v \in V$ occurs as a representative at most $8 n^{2/3}\log^{1/2}n$ times.
In summary, we have found two distinct representatives for each $C \in \mathcal{C}$ such that each $v \in V$ occurs as a representative at most $16 n^{2/3}\log^{1/2}n$ times.

Now combining the representatives of the long cycles and the short cycles, we see that each vertex is represented at most $8n^{2/3} + 16 n^{2/3}\log^{1/2}n \leq 24 n^{2/3}\log^{1/2}n$ times.
\end{proof}

For the remainder of the section, assume $T=(V,E)$ is a tournament with $\disc(T) > 7n^{17/9 + \gamma}$ where $1/n \ll \gamma$. In the next two lemmas, we will construct paths in $T$ that will form a partial decomposition of $T$ when combined in the right way. Moreover, it will turn out that these paths can also be used to ``absorb'' cycles; this is the crucial idea of the proof of Theorem~\ref{thm:LPST1}.

For any digraph $D$, any $s \in \mathbb{R}$, and $* \in \{+,-\}$, we define $W^*_s(D):=\{v \in V(D): \disc^*_D(v) \geq s \}$.

\begin{lemma} \label{lma:shortpathsatv(new)}
Let $n \in \mathbb{N}$ and $1/n \ll \gamma$.
Suppose that $T=(V,E)$ is a tournament on $n$ vertices with $\disc(T) \geq 8 n^{17/9+\gamma}$. Set $s = n^{8/9 + \gamma}$.
Let $H \subseteq T$ with $\Delta(H) \leq s$ and $S \subseteq V$ with $|S| \leq s$. 
For any $v \in V \setminus S$, there exist $n^{2/3+\gamma}$ paths in $T- H - S$ that start in $W^+_{s}$, end at $v$, have length at most $4n^{1/9}$, and are vertex-disjoint except at their end-point $v$.
\end{lemma}

In the statement above, a path could be a single vertex $v$ and in that case we think of $v$ as being vertex disjoint with itself except at its endpoint.

Note that, by symmetry, the same result as above holds if we wish to find paths from $v$ to $W^-_s$. 
\begin{proof}
We write $W^+$ for $W^+_{s}$ and note that $\disc(T) \leq |W^+|n + ns$ so that $|W^+| \geq 7n^{8/9 + \gamma}$.
If $v \in W^+$ then we are done (by the remark above), so assume not. 
Write $T' := T - H - S$. 
Let $A^+ := W^+ \setminus S$.
Suppose that all $(A^+,v)$-separators in~$T'$ have size at least $s$. 
Thus by Menger's Theorem we can find at least $s$ paths in $T'$ that start in $A^+$, end at~$v$ and are vertex disjoint except for their common 
endpoint $v$. 
If we pick the shortest $n^{2/3+\gamma}$ of these paths, they all have length at most $4n^{1/9}$ (since otherwise we  have at least $s - n^{2/3} \geq \frac{1}{2}n^{8/9}$ paths of length at least $4n^{1/9}$ that are vertex-disjoint except for one common vertex; such paths cover at least $\frac{1}{2}n^{8/9} \cdot (4n^{1/9} - 1) > n$ vertices, a contradiction).
Therefore to prove the lemma, it suffices to show that all $(A^+,v)$-separator~$X$ in $T'$ satisfy $|X| \ge n^{8/9+\gamma}$.

Let $X$ be a $(A^+,v)$-separator in $T'$ and let $\hat{T} := T' - X = T - H - (S \cup X)$. 
Define
\[
B = \{x \in V(\hat{T}): \exists \text{ a path from } x \text{ to } v \text{ in } \hat{T}   \}.
\] 
Then $A^+ \cap B = \emptyset$ (since otherwise $X$ is not a $(A^+,v)$-separator) and so $W^+ \cap B = \emptyset$. Furthermore, by the definition of $B$ there are no 
directed edges in $\hat{T}$ from $\overline{B} := V(\hat{T}) \setminus B$ to~$B$. Using this and the fact that $T$ is a tournament we have for all $x \in B$ that
\begin{align*}
|N_T^+(x) \setminus B| 
&\geq V(T) - |B| - |X| - |S| - \Delta(H) \\
&\geq |W^+| -|X| - |S| - \Delta(H),
\end{align*}
and
\[
|N_T^-(x) \setminus B| \leq |X| +|S| + \Delta(H).
\]
Pick a vertex~$x^* \in B$ with $\disc_{T[B]}(x^*) \geq 0$ (note that 
every directed graph has a vertex with non-negative 
excess). Then we have 
\begin{align*}
\disc_T(x^*) & \ge \disc_{T[B]}(x^*) + |N_T^+(x^*) \setminus B| -  |N_T^-(x^*) \setminus B|  \\
&\geq 0 + |W^+| - 2|X| -  2 |S| - 2\Delta(H) 
\ge |W^+| - 4s -2|X|.
\end{align*} 
We know that $\disc_T(x^*) \leq s$ (otherwise $x^* \in W^+$, a contradiction) and that $|W^+| \geq 7 n^{8/9+\gamma} = 7s$. 
Hence $|X| \geq s = n^{8/9+\gamma}$, 
as required.
\end{proof}

By inductively applying the previous lemma, we obtain the following.

\begin{lemma} \label{lma:shortpathseverywhere(new)}
Let $n \in \mathbb{N}$ and $1/n \ll \gamma$.
Suppose that $T=(V,E)$ is a tournament on $n$ vertices with $\disc(T) \geq 8 n^{17/9+\gamma}$.
Let $\ell := n^{2/3+\gamma}$ and $m := 4n^{1/9}$ and $s:= n^{8/9 + \gamma}$.
Then we can find edge-disjoint paths $P^v_j, Q^v_j$ where $v \in V$, $j=1, \ldots, \ell$ with the following properties:
\begin{enumerate}[label = {\rm (\roman*)}]
	\item $P^v_j$ is a path of length at most $m$ from $W^+_{s}$ to $v$ and $Q^v_j$ is a path of length at most $m$ from $v$ to $W^-_{s}$;
	\item for each fixed $v \in V$, the paths $P^v_1, \ldots, P^v_{\ell}$ are vertex-disjoint except that they all meet at $v$ and the paths $Q^v_1, \ldots, Q^v_{\ell}$ are vertex-disjoint except that they all meet at $v$;
	\item $\Delta ( \bigcup _{ v,j } (P^v_j \cup Q^v_j)) < n^{8/9+\gamma} =s$.
\end{enumerate} 
\end{lemma}
\begin{proof}
Fix an ordering $v_1, \ldots, v_n$ of the vertices of $T$ and inductively construct the desired paths as follows. Suppose at the $k$th step, we have constructed the $P^{v_{i}}_j$ and $Q^{v_{i}}_j$ for all $i\leq k-1$ and all $j \le \ell $ satisfying the first two conditions of the lemma.
Furthermore, we assume that the oriented graph $H_{k-1}$ on $V$, which is union of the paths constructed so far, satisfies
\begin{align} \label{path-induction1'}
d_{H_{k-1}}(v_i) &\leq 2\ell + s/2   &&\forall 1 \le i \le k-1 
\\ \label{path-induction2'} \text{ and }  
d_{H_{k-1}}(v_i) &\leq s/2  &&\forall k \leq i \leq n.
\end{align}

By our choice of parameters, we have 
\begin{equation}
\label{eqn:DeltaH}
\Delta(H_{k-1}) \leq 2 \ell + s/2  <  s.
\end{equation}
Let $S^*$ be the set of vertices $v \in V$ such that $d_{H_{k-1}}(v) \ge s/4 $.
Note that $\frac14 s |S^*| \le 2|E(H_{k-1})| \le 4 n m \ell =  16 n^{16/9+\gamma}  $, so $|S^*| \leq 64n^{8/9} \le s $.

Now applying Lemma~\ref{lma:shortpathsatv(new)} (where $(H_{k-1}, S^*)$ play the role of $(H,S)$), we obtain vertex-disjoint (except at $v_k$) paths  $P^{v_{k}}_j$ for all $j \le \ell$ from $W_s^+$ to $v_k$ each of length at most~$m$.
Applying Lemma~\ref{lma:shortpathsatv(new)} again (where $(H_{k-1} \cup (\bigcup_j E(P^{v_k}_j), S^*)$ play the roles of $(H,S)$ and noting $\Delta (\bigcup_j E(P^{v_k}_j) \le \ell$), we obtain vertex-disjoint (except at $v_k$) paths  $Q^{v_{k}}_j$ for all $j \le \ell$ from $v$ to $W^-_{s}$, each of length at most $m$. Note that all the new paths are edge-disjoint from each other and from the old ones and satisfy conditions (i) and (ii) of the lemma. 

Letting $H_k$ be the union of all the paths constructed so far, note that compared to $H_{k-1}$, the degree of $v_k$ goes up by at most $2\ell$ and the degree of every vertex $v \in V \setminus ( S^* \cup v_k)$ goes up by at most $4$. 
Thus (\ref{path-induction1'}) and (\ref{path-induction2'}) hold. At the $n$th step we are able to construct all the paths satisfying properties (i) and (ii), and property (iii) also holds by~(\ref{eqn:DeltaH}). 
\end{proof}

We now prove the following theorem which immediately implies Theorem~\ref{thm:LPST1} by taking $\eps = 1/18$.

\begin{theorem}
\label{lma:Exactdec-highdiscexp(new)}
Let $n \in \mathbb{N}$ and $1/n \ll \gamma$.
Suppose that $T=(V,E)$ is a tournament on $n$ vertices with $\disc(T) \geq 8 n^{17/9+\gamma}$.
Then $T$ has a perfect decomposition.
\end{theorem}
\begin{proof}
Let $\ell : = n^{2/3+\gamma} $ and $m:= 4n^{1/9}$ and $s := n^{8/9 + \gamma}$.
Apply Lemma~\ref{lma:shortpathseverywhere(new)} to $T$ so that 
we obtain edge-disjoint paths $P^v_j, Q^v_j$, where $v \in V$ and $j=1, \ldots, \ell$ with the following properties:
\begin{enumerate}[label = {\rm (\roman*)}]
	\item $P^v_j$ is a path of length at most $m$ from $W^+_{s}$ to $v$ and $Q^v_j$ is a path of length at most $m$ from $v$ to $W^-_{s}$;
	\item for each fixed $v \in V$, the paths $P^v_1, \ldots, P^v_{\ell}$ are vertex-disjoint except that they all meet at $v$ and the paths $Q^v_1, \ldots, Q^v_{\ell}$ are vertex-disjoint except that they all meet at $v$;
	\item $\Delta ( \bigcup _{ v,j } (P^v_j \cup Q^v_j)) < s$.
\end{enumerate} 

Call a path of the form $P^v_j$ a \emph{$v$-in-path} and a path of the form $Q^v_j$ a \emph{$v$-out-path}.
Write $H$ for the graph that is the union of these paths and let $T' = T - H$. For each $v \in V$ and $j \le \ell$, each walk $P^v_j \cup Q^v_j$ starts in $W^+_{s}$ and ends in $W^-_s$ and no vertex occurs as a start or end point  more than $s$ times. Therefore we have that  $\disc^{\pm}_H(v) \leq s \leq \disc^{\pm}_T(v)$ for all $v \in W^{\pm}_s$ and $\disc_H(v) = 0$ for all other $v \in V \setminus (W_s^+ \cup W_s^-)$. This means in particular that $\disc(T) = \disc(H) + \disc(T')$ (by Proposition~\ref{pr:disc}). (In fact, $\disc(H) = \ell n = n^{5/3+\gamma}$.)
Let $T_E$ be a maximal Eulerian subgraph of $T'$ and let $T_R = T' - T_E$, where $T_R$ is necessarily acyclic. Thus we have that $T = H \cup T_R \cup T_E$ and $\disc(T) = \disc(H) + \disc(T_R) + \disc(T_E)$ (and of course $\disc(T_E) = 0$).

Finally we show how to decompose $T_E \cup H$ into $\disc(H)$ paths.
Apply Lemma~\ref{lma:EulerDecomp(new)} to $T_E$.
Thus we can decompose $T_E$ into $t \le 50 n^{4/3} \log n$ cycles $C_1, \ldots, C_t$ and for each cycle $C_i$ we can find distinct representatives $x^i_1, x^i_2, \ldots, x^i_{r_i} \in V(C_i)$ (indexed in order) with the following properties:
 \begin{enumerate}[label = {\rm (\roman*$'$)}]
	\item every cycle has at least two representatives, i.e. $r_i \geq 2$ for all $i$;
	\item the interval between consecutive vertices on a cycle $x^i_{j}C_i x^i_{j+1}$ has length at most $n^{2/3}$;
	\item every vertex $v \in V$ occurs as a representative at most $24 n^{2/3}\log^{1/2}n$ times.
\end{enumerate} 
Write $C^{i}_j$ for the interval $x^i_{j}C_i x^i_{j+1}$.
By (i), (ii), (ii$'$) and (iii$'$), for each $i \le t$ and $j \le r_i$, we can greedily find  distinct $P^{x^i_j}_{j'}$ such that each $P^{x^i_j}_{j'} C^i_j$ is a path from $W^+_{s}$ to $x^i_{j+1}$. (Given $C^i_j$, since the paths $P^{x^i_j}_{1}, \ldots, P^{x^i_j}_{\ell}$ are vertex disjoint (except at $x^i_j$), at least $\ell - |C^i_j|$ of these paths avoid $C^i_j$ and since we never use more than $24 n^{2/3}\log^{1/2}n$ of these paths, there is always one available.) 
Hence we have shown that $\bigcup_{v,j} P^v_j \cup T_E$ can be edge-decomposed into $\ell n$ paths $P_1, \dots, P_{\ell n}$ each of length at most $n^{2/3} + m$.
Notice crucially that each vertex $v$ is an end point of exactly $\ell$ paths and at least $\ell - 24 n^{2/3}\log^{1/2}n$ of such paths belong to $\{P^v_j: j \le \ell \}$.

We now extend $P_1, \dots, P_{\ell n}$ using the paths $\{Q^v_j: v \in V,j \le \ell \}$ as follows.
Consider any $v \in V$. 
Let $\mathcal{P}_v$ be the set of $P_i$ with end point~$v$ and let $\mathcal{Q}_v = \{Q^v_j: j \le \ell \}$. 
Clearly, $|\mathcal{P}_v| = \ell = |\mathcal{Q}_v|$.
Let $\mathcal{P}'_v $ (and $\mathcal{Q}'_v)$ be the largest set of vertex-disjoint  (except at $v$) paths of $\mathcal{P}_v$ (and $\mathcal{Q}_v$, respectively).
Thus $|\mathcal{P}'_v| \ge \ell - 24 n^{2/3}\log^{1/2}n$ and $|\mathcal{Q}_v'| = \ell$.
Let $B$ be the auxiliary bipartite graph with vertex partition $\mathcal{P}_v$ and $\mathcal{Q}_v$, where $P \in \mathcal{P}_v$ is joined to $Q \in \mathcal{Q}_v$ if and only if $V(P) \cap V(Q) = \{v\}$. 
For each $Q \in \mathcal{Q}_v$, $|N_B(Q)| \ge |\mathcal{P}'_v| - |V(Q)| \geq \ell - (24 n^{2/3}\log^{1/2}n ) - m \ge \ell/2$.
Similarly, we have $|N_B(P)| \ge  \ell /2$ for each $P \in \mathcal{P}_v$.
Hence $B$ has a perfect matching, which implies $\bigcup\mathcal{P}_v \cup \bigcup\mathcal{Q}_v$ can be decomposed into $\ell$ paths. 
Therefore, $T_E \cup H = \bigcup_{v \in V} (\bigcup\mathcal{P}_v \cup \bigcup\mathcal{Q}_v)$ can be decomposed into $\ell n = \disc(H)$ paths.

Thus we can now write $T = H \cup T' = (H \cup T_E) \cup T_R$ where $\disc(T) = \disc(H) + \disc(T') = \disc(H) + \disc(T_R) = \ell n + \disc(T_R)$ and where $H \cup T_E$ can be decomposed into $ \ell n$ paths and $T_R$ can be decomposed into $\disc(T_R)$ paths (by Proposition~\ref{pr:acyclic}). Hence $T$ can be decomposed into $\disc(T)$ paths.
\end{proof}

\section{Further preliminaries and overview}
\label{se:fprelims}

In this section we provide further preliminaries used in Sections~\ref{se:W} and \ref{se:finaldecomp} as well as an overview of the proof of Theorem~\ref{thm:LPST2}. 

\subsection{Partial decompositions}

We will use the following easy facts about partial decompositions repeatedly. The proofs are straightforward, but we give them for completeness.

\begin{proposition}
\label{pr:PartDecomp1}
Let $D$ be a directed graph and let $\mathcal{P} = \{P_1, \ldots, P_k \}$ be a partial decomposition of $D$ where $P_i$ is a path from $x_i$ to $y_i$. Then the following hold.
\begin{itemize}
\item[(a)] Any $\mathcal{Q} \subseteq \mathcal{P}$ is a partial decomposition of $D$ and a partial decomposition of $D - E(\mathcal{P} \setminus \mathcal{Q})$.
\item[(b)] If $\mathcal{Q}$ is a partial decomposition of $D - E(\mathcal{P})$ then $\mathcal{P} \cup \mathcal{Q}$ is a partial decomposition of $D$ (and hence so is $\mathcal{Q}$).
\item[(c)] If $\pi$ is a permutation of $[k]$ and $\mathcal{Q} = \{Q_1, \ldots, Q_r \}$ is  a set of paths with $r \leq k$ and $Q_i$ is a path from $x_i$ to $y_{\pi(i)}$, then $\mathcal{Q}$ is a partial decomposition of $D$. 
\item[(d)] If $D' \subseteq D$ is an Eulerian subdigraph of $D$ and $\mathcal{Q}$ is a partial decomposition of $D - D'$, then $\mathcal{Q}$ is a partial decomposition of $D$.
\end{itemize}
\end{proposition}
\begin{proof}
For any collection of paths $\mathcal{A} = \{A_1, \ldots, A_t \}$ where $A_i$ is a path in a digraph $D$ and $x \in V(D)$, write $p^+_{\mathcal{A}}(x)$ for the number of paths in $\mathcal{A}$ that start at $x$ and  $p^-_{\mathcal{A}}(x)$ for the number of paths in $\mathcal{A}$ that end at $x$.

(a) The fact that $\mathcal{Q}$ (and $\mathcal{P} \setminus \mathcal{Q}$) is a partial decomposition of $D$ is immediate. For the second part note that for any $x \in V(D)$, if $\disc_D(x) \geq 0$ then
\begin{align*}
\disc^+_{D - E(\mathcal{P} \setminus \mathcal{Q})}(x) 
&= \disc^+_D(x) - p^+_{\mathcal{P} \setminus \mathcal{Q}}(x) 
+ p^-_{\mathcal{P} \setminus \mathcal{Q}}(x) \\
&\geq p^+_{\mathcal{P}}(x) - p^+_{\mathcal{P} \setminus \mathcal{Q}}(x) + 0 
= p^+_{\mathcal{Q}}(x),
\end{align*}
where the inequality holds since $\mathcal{P}$ is a partial decomposition of $D$. A similar statement holds if $\disc_D(x) \leq 0$.

(b) Note that for any $x \in V(D)$, if $\disc_D(x) \geq 0$ then
\[
 \disc_D(x) - p^+_{\mathcal{P}}(x) + p^-_{\mathcal{P}}(x)
= \disc^+_{D - E(\mathcal{P} )}(x) 
  \geq p^+_{\mathcal{Q}}(x)
\] 
and a similar statement holds if $\disc_D(x) \leq 0$. Rearranging gives $\disc_D(x) \geq p^+_{\mathcal{P}}(x) + p^+_{\mathcal{Q}}(x) = p^+_{\mathcal{P} \cup \mathcal{Q}}(x)$.

(c) Here we note that $p^+_{\mathcal{P}}(x) \geq p^+_{\mathcal{Q}}(x)$ and $p^-_{\mathcal{P}}(x) \geq p^-_{\mathcal{Q}}(x)$ for all $x \in V(D)$.
 
(d) Here we note that $\disc_D(x) = \disc_{D - E(D')}(x)$ for all $x \in V(D)$.
\end{proof}

\begin{proposition}
\label{pr:PartDecomp2}
Let $D$ be a directed graph and suppose there is a partition of $V(D)$ into sets $A^+, A^-, R$ such that $E_D(R, A^+) = E_D(A^-, R) = E(D[A^+ \cup A^-]) = \emptyset$. Then the following holds.
\begin{itemize}
\item[(a)] If $\mathcal{P}= \{P_1, \ldots, P_r \}$ is a partial decomposition of $D[R]$, then there is a partial decomposition $\mathcal{P}' = \{P_1', \ldots, P_r' \}$ of $D$ such that $V(P_i') \cap R = V(P_i)$ for all $i = 1,\ldots, r$.
\item[(b)] If there is a perfect decomposition of $D[R]$ then there is a perfect decomposition of $D$. 
\item[(c)] If in addition we assume that $\disc_D(v) \geq 0$ for every $v \in N^+_D(A^+)$ and $\disc_D(v) \leq 0$ for every $v \in N^-_D(A^-)$ and $N^+_D(A^+) \cap N^-_D(A^-) = \emptyset$ then  $\disc(D[R]) = \disc(D)$.
\end{itemize}
\end{proposition}
\begin{proof}
%
%
(a) This is easily proved by induction on the number of paths; we give the details for completeness. 
By induction we will find paths $P_1' \ldots, P_{r}'$ such that each path $\{ P_i' \}$ is a partial decomposition of $D_i := D - (P_1' \cup \cdots \cup P_{i-1}')$ for $i = 1, \ldots, r$ and $V(P_i) = V(P_i') \cap R$. By $r$ applications of Proposition~\ref{pr:PartDecomp1}(b), $\{P_1' \ldots, P_r'\}$ is a partial decomposition of $D$ with the desired properties.

Suppose we have found the paths $P_1' \ldots, P_{k-1}'$ as described above.
Then $\{P_k\}$ is a partial decomposition of $D[R] - (P_1 \cup \cdots \cup P_{k-1}) = D_k[R]$. Write $P_k = xP_ky$. If there is some edge $a^+x \in E(D_k)$ with $a^+ \in A^+$ then append it to $P_k$ and if there is some edge $ya^-\in E(D_k)$ with $a^- \in A^-$ then append it to $P_k$ and  write $P_k'$ for the resulting path. Let $P_k' = x'P_ky'$; we show that $\disc_{D_k}(x')>0>\disc_{D_k}(y')$ proving that $\{P_k'\}$ is a partial decomposition of $D_k$.

By symmetry it is sufficient to show $\disc_{D_k}(x')>0$.  
If $x' = a^+ \in A^+$ then this is certainly the case. If $x' = x$ then   
\[
\disc_{D_k}(x) \geq \disc_{D_k[R]}(x)>0. 
\]
The first inequality holds because there is no edge 
in $D_k$ from $A^+$ to $x$ (nor from $A^-$ to $x$ from the statement of the lemma). The second inequality holds 
because $P_k$ starts at $x$ and $\{ P_k \}$ is a partial decomposition of $D_k[R]$.

(b) From (a) we can extend our perfect decomposition of $D[R]$ to a partial decomposition $\mathcal{Q}_1$ of $D$ that uses every edge of $D[R]$. The remaining digraph $D - E(\mathcal{Q}_1) \subseteq D - E(D[R])$ is acyclic so has a perfect decomposition $\mathcal{Q}_2$ by Proposition~\ref{pr:acyclic}. Therefore $\mathcal{Q}_1 \cup \mathcal{Q}_2$ is a perfect decomposition of $D$ by Proposition~\ref{pr:PartDecomp1}(b).

(c) This is proved by induction on the number of edges between $A^+ \cup A^-$ and $R$. If $D$ has no edges between $A^+ \cup A^-$ and $R$ then we are done. For any edge $e = a^+r$ with $a^+ \in A^+$ and $r \in R$, $\disc_{D-e}(r) > \disc_D(r) \geq 0$. Furthermore $r \not\in N_{D-e}^-(A^-)$ because $N^+_D(A^+) \cap N^-_D(A^-) = \emptyset$. It is easy to check that the conditions in (c) are satisfied for $D-e$ so we can assume by induction that $\disc(D-e) = \disc(D[R])$. Also, we see that adding the edge $e$ back to $D-e$ reduces $\disc(r)$ by $1$ and increases $\disc(a^+)$ by $1$ so that $\disc(D) = \disc(D-e) = \disc(D[R])$. The case when $e =ra^-$ for some $r \in R$ and some $a^- \in A^-$ holds similarly. 
\end{proof}

\subsection{Robust expanders}

Here we introduce the basic notions of robust expansion and their consequences, which we will use in Sections~\ref{se:W} and~\ref{se:finaldecomp}. Most of this can be found in \cite{KO,KO2}

We give the definition of robust expander here for completeness. We will not use the definition directly, but only use some of the consequences given below.

\begin{definition}
An $n$-vertex digraph $D$ is a robust $(\nu, \tau)$-outexpander if for every $S \subseteq V(D)$ with $\tau n \leq |S| \leq (1- \tau)n$ there is some set $T \subseteq V(D)$ with $|T| \geq |S| + \nu n$ such that every vertex in $T$ has at least $\nu n$ in-neighbours in $|S|$.
%
%
\end{definition}

It turns out that sufficiently dense oriented graphs are robust expanders.
\begin{lemma}[{\cite[Lemma 13.1]{KO}}]
 \label{lma:outexpander}
Let $0 < 1/n \ll \nu \ll \tau \ll \eps$.
Suppose that $D$ is an oriented graph on $n$ vertices with $\delta^0 (D) \ge (3/8+\eps) n $.
Then $D$ is a robust $(\nu,\tau)$-outexpander.
\end{lemma}

The notion of robust expansion was developed to help solve problems on Hamilton cycles. Here are two of the main results.

\begin{theorem}[{\cite[Corollary~6.9]{klos}}]
\label{thm:robexp}
Let $0 < 1/n \ll \eps \ll \nu \ll \tau \ll \delta$. 
Suppose that $D$ is a robust $(\nu,\tau)$-outexpander on $n$ vertices with $\delta^0(D) \geq \delta n$.
Let $a,b \in V(D)$.
Then $D$ contains a Hamilton path from $a$ to~$b$.
\end{theorem}

\begin{theorem}[\cite{KO}]
\label{thm:kelly}
Let $0<1/n \ll \nu \ll  \tau \ll \delta$. 
Suppose that $D$ is an $r$-regular oriented graph with $r \ge  \delta n$ and a robust $(\nu, \tau )$-outexpander.
Then $E(D)$ can be decomposed into $r$ edge-disjoint Hamilton cycles. 
\end{theorem}

An immediate consequence of the above is the following path decomposition result, which we use right at the end of the paper.
\begin{theorem}
\label{thm:KellyPaths}
Let $0 < 1/n \ll 1$ and let $D$ be an oriented graph with a vertex partition $V(D) = X^+ \cup X^- \cup X^0$ with $|X^+| = |X^-| = d \geq 3n/7$  such that 
\begin{align*}
\disc_{D}(v) = 
\begin{cases}
1 &\text{if } v \in X^+; \\
0 &\text{if } v \in X^0; \\
-1 &\text{if } v \in X^-;
\end{cases}
\:\:\:\text{ and }\:\:\:
d_{D}(v) = 
\begin{cases}
2d-1 &\text{if } v \in X^+; \\
2d &\text{if } v \in X^0; \\
2d-1 &\text{if } v \in X^-.
\end{cases}
\end{align*}
Then $D$ has a perfect decomposition.
\end{theorem}
\begin{proof}
Fix $\nu, \tau$ such that $1/n \ll \nu \ll \tau \ll 1$.
We form $D'$ by adding a vertex $y$ such that $N^+_{D'}(y) = X^+$ and $N^-_{D'}(y) = X^-$. Then $D'$ is a regular oriented graph with in- and outdegree $d > 3/7n$ and
so is a robust $(\nu, \tau)$-outexpander by Lemma~\ref{lma:outexpander}. Thus it
  has an edge decomposition into Hamilton cycles $H_1, \ldots, H_d$ by Theorem~\ref{thm:kelly}. Taking $P_i$ to be the path $H_i - y$, $\mathcal{P} = \{ P_1, \ldots, P_d \}$ gives a perfect decomposition of $D$. 
\end{proof}

Robust expanders are highly connected as one would expect and so we can find (many) short paths between any pair of vertices. This is made precise in the following three lemmas.
\begin{lemma}
[see e.g.\ {\cite[Lemma 9]{LP}}]
 \label{lma:shortpath}
Let $n \in \mathbb{N}$ and $0 < 1/n \ll \nu \ll \tau \ll \delta \leq 1$.
Suppose that $D$ is a robust $(\nu,\tau)$-outexpander on $n$ vertices with $\delta^0(D) \geq \delta n$.
Then, given any distinct vertices $x,y \in V(D)$, there exists a path $P$ in $D$ from $x$ to $y$ such that $|V(P)| \leq \nu^{-1}$.
\end{lemma}

The following lemma and its corollary will be used many times in our proof. 

\begin{lemma} \label{lma:manyshortpaths3}
Let $n \in \mathbb{N}$ and $0 < 1/n \ll \gamma \ll 1$.
Suppose that $D$ is an oriented graph on $n$ vertices with $\delta^0(D) \ge 3n/7$.
Let $H_1, \dots, H_m$ be directed multigraphs on~$V(D)$ with $\Delta(H_i) \leq 2$, $|E(H_i)| \le \gamma n$  and $m \le \gamma n$. 
Let $S_1, \dots, S_m \subseteq V(D)$ with $|S_i| \le  n/25$ and $S_i \cap V(E(H_i)) = \emptyset$. 
Then there exists a set of edge-disjoint paths $\mathcal{P} = \{P_{i,e} \colon i \in [m]$ and $ e \in E(H_i)\}$ in~$D$ such that 
\begin{enumerate}[label = {\rm (\roman*)}]
	\item $P_{i,e}$ has the same starting and ending points as~$e$;
	\item the paths in $\mathcal{P}_i := \{P_{i,e} \colon e \in E(H_i)\}$ are internally vertex-disjoint;
	\item $V(\cup \mathcal{P}_i) = V(D) \setminus S_i$;
	\item $\Delta(\cup \mathcal{P}) \le 2 m $. 
\end{enumerate} 
\end{lemma}

\begin{proof}
We proceed by induction on~$m$, the number of multigraphs. 
Suppose that we have already found $\mathcal{P}' := \{P_{i,e} \colon i \in [m-1] $ and $ e \in E(H_i) \}$ with the desired properties. 

Let $\nu, \tau, \eps$ be such that $\gamma \ll \nu \ll \tau \ll \eps \ll 1$.
Pick an arbitrary ordering $e_1, \ldots, e_r$ of the edges in $E(H_m)$.
Further assume that for some $j \in [r]$, we have already constructed paths $P_{1}, \dots, P_{j-1}$ such that, for each $j' \in [j-1]$, 
\begin{enumerate}[label = {\rm (\roman*)}]
	\item $P_{j'}$ has the same starting and ending points as~$e_{j'}$ and has length at most $\nu^{-1} $;
	\item $V(E(H_i)),S_i, V(P_{1})\setminus V(e_{1}), \dots, V(P_{j-1})\setminus V(e_{j-1})$ are disjoint.
\end{enumerate}
We now find $P_{j}$ as follows.
Let $e_j = xy$. 
Let $D' : = D - E(\cup \mathcal{P}') - S_i - (V (P_1 \cup \dots \cup P_{j-1}) \setminus \{x,y\} )$.
Since $|S_i \cup V (P_1 \cup \dots \cup P_{j-1}) )| \le (\frac{1}{25} + \nu^{-1}\gamma) n$,
then $|D'| \ge (1 - \frac{1}{25} - \nu^{-1}\gamma) n $ and $\delta^0(D') \ge \delta^0(D) - (\frac{1}{25} + \nu^{-1}\gamma) n \geq (3/8 + \eps)  |D'|$.
By Lemma~\ref{lma:outexpander}, $D'$ is a robust $(\nu,\tau)$-outexpander.
If $j <r$, then  $D'$ has a path $P_{j}$ from $x$ to $y$ of length at most $\nu^{-1}$ by Lemma~\ref{lma:shortpath}.
If $j=r$, then $D'$ has a Hamilton path $P_{j}$ from $x$ to $y$ by Theorem~\ref{thm:robexp}.
We are done by setting $P_{m,e_j}:=P_{j}$ for all $j \in [r]$. 
\end{proof}

Let $H$ be a directed multigraph on $n$ vertices with $\Delta(H) \leq \gamma n$. 
Note that $H$ can be decomposed into digraphs $H_1, \dots, H_m$ with $m \le 2\sqrt{\gamma} n$ and $\Delta(H_i) \le 1$ and $|E(H_i)| \le 2\sqrt{\gamma} n$.
(By Vizing's theorem, $H$ can be partitioned into $(\gamma n)+1$ matchings and each matching can then be further split into $\gamma^{-1/2}$ almost equal parts to give us the $H_i$.)
Applying the previous lemma to these $H_i$, we obtain the following corollary. 

\begin{corollary} \label{cor:manyshortpaths3}
Let $n \in \mathbb{N}$ and $0 < 1/n \ll \gamma \ll 1$.
Suppose that $D$ is an oriented graph on $n$ vertices with $\delta^0(D) \ge 3n/7$.
Let $H$ be a directed multigraph on~$V(D)$ with $\Delta(H) \leq \gamma n$. 
Then there exists a set of edge-disjoint paths $\mathcal{P} = \{P_{e} \colon e \in E(H)\}$ in~$D$ such that 
\begin{enumerate}[label = {\rm (\roman*)}]
	\item $P_{e}$ has the same starting and ending points as~$e$;
	\item $\Delta(\cup \mathcal{P}) \le 4 \sqrt{\gamma} n $. 
\end{enumerate} 
\end{corollary}

\subsection{Overview}
\label{se:overview}

In this subsection, we give an overview of the proof of Theorem~\ref{thm:LPST2} (which is proved in Sections~\ref{se:W} and~\ref{se:finaldecomp}).
We wish to show that every even $n$-vertex tournament $T$ satisfying $\disc(T)>Cn$ and $n$ sufficiently large has a perfect decomposition (i.e.\ is consistent). Let us fix such a tournament $T$; we may further assume by Theorem~\ref{thm:LPST1} that $\disc(T) < n^{2 - \eps}$. We will accomplish this in three steps. In each step we reduce the problem of finding a perfect decomposition of $T$ to the problem of finding a perfect decomposition of a digraph that looks more and more like the digraph described in Theorem~\ref{thm:KellyPaths}.

\vspace{0.3 cm}
\noindent
{\bf Step 1 - remove vertices of high excess.} 
Let $W = \{v \in V(T): |\disc(v)| > \alpha n \}$ for some suitable $\alpha$. Note that since $\disc(T)$ is small, $W$ is also small. Let $W^{\pm}$ be respectively the vertices of $W$ with positive / negative excess and let $R = V(T) \setminus W$. We will construct a partial decomposition $\mathcal{P}_0$ of $T$ with a small number of paths that uses all edges in $E_T(R,W^+) \cup E_T(W^-,R) \cup E_T(W)$ but does not interfere much with $E_T(R)$. 
Set $D_1 = T - \cup \mathcal{P}_0 - W$. Now we can apply  Proposition~\ref{pr:PartDecomp2}(b) to $T - \cup \mathcal{P}_0$ to conclude that if $D_1$ has a perfect decomposition, then so does $T- \cup \mathcal{P}_0$ and hence so does $T$. Thus we have reduced the problem of finding a perfect decomposition of $T$ to that of finding one for $D_1$, but where $D_1$ has no vertices of high excess and 
\[
\disc(D_1) = \disc(T - \cup{P}_0) = \disc(T) - |\mathcal{P}_0| \geq C'n. 
\]
Since there are no vertices of high excess, $D_1$ is close to regular and so one can apply the methods of robust expansion. This step takes place in Theorem~\ref{thm:RemHighDisc} and the key tool for finding $\mathcal{P}_0$ is Lemma~\ref{lma:W} from Section~\ref{se:W}.

\vspace{0.3 cm}
\noindent
{\bf Step 2 - equalise the number of vertices of positive and negative excess.} Given $D_1$ from the previous step, it may be the case that almost all vertices of $D_1$ have say negative excess that is $U^-(D_1)$ is significantly larger than $U^+(D_1)$, where $U^{\pm}(D)$ denote the set of vertices of positive / negative excess in $D$. 

For some fixed $z \in U^-(D_1)$ consider how we might change the sign of its excess. The idea would be to find $x \in U^+(D_1)$ with $xz \in E(D_1)$ and a partial decomposition $\mathcal{Q}$ that
\begin{itemize}
\item has a path $Q^*$ that starts at $x$, uses the edge $xz$ but does not end at $z$;
\item uses all edges incident with $x$;
\item has exactly $\disc^-(z)$ paths ending at $z$.
\end{itemize}
If we can find such a $\mathcal{Q}$, then consider $D_1' = D_1 - E(\mathcal{Q} \setminus \{Q^*\}) - x$. We have $\disc_{D_1'}(z) = 1$ and moreover if $D_1'$ has a perfect decomposition, so does $D_1$ (the path that starts at $z$ in a perfect decomposition of $D_1'$ would be extended by the edge $xz$ in $D_1$). 

We refine this idea to switch the sign of the excess for many vertices in $U^-(D_1)$ in Theorem~\ref{thm:DiscBalance}. We carefully choose a small set of vertices $X \subseteq U^+(D_1)$ and a suitably larger set $Z \subseteq U^-(D_1)$ and a partial decomposition $\mathcal{P}_1$ of $D_1$ such that writing $D_2 = D_1 - E(\mathcal{P}_1) - X$, $D_1$ has a perfect decomposition if $D_2$ does, and $U^+(D_2) = U^-(D_2) \cup Z \setminus X$ and $U^-(D_2) = U^-(D_1) \setminus Z$. Again we use Lemma~\ref{lma:W} from Section~\ref{se:W} as a tool.

\vspace{0.3 cm}
\noindent
{\bf Step 3 - control the degrees.} In this final step (Theorem~\ref{thm:FinalDecomp}), starting with $D_2$ we carefully construct a partial decomposition $\mathcal{P}_2$ of $D_2$ such that $D_3 = D_2 - E(\mathcal{P}_2)$ is a digraph satisfying the properties of Theorem~\ref{thm:KellyPaths}. Hence $D_3$ has a perfect decomposition, and thus so does $D_2$, $D_1$, and $T$. 

\vspace{0.3 cm}
\noindent
We make use of the robust expansion properties of $D_2$ to construct $\mathcal{P}_2$; this is why we need step 1. Also, essentially by definition, the excess of a vertex can never change sign when we remove a partial decomposition from a digraph; this is why we need step 2. Each of steps 1 and 2 will require us to remove a partial decomposition of size linear in $n$, and this is why we must start with $\disc(T) > Cn$ for a suitably large $C$.

\section{Removing small vertex subsets}
\label{se:W}

In Section~\ref{se:highdisc}, we showed how to find a perfect decomposition of $n$-vertex tournaments $T$ ($n$ even) whenever $\disc(T) > n^{2 - \eps}$.
For the remaining cases of Thoerem~\ref{thm:LPST1}, we will require a preliminary result which we prove in this section.
 For almost complete oriented graphs $D$ satisfying $Cn \leq \disc(D) \leq n^{2 - \eps}$, we show in Lemma~\ref{lma:W} that for certain choices of small $W \subseteq V(D)$, we can find a partial decomposition $\mathcal{P}$ of $D$ that uses all the edges incident with $W$ going in the ``wrong'' direction. We will also guarantee that $\mathcal{P}$ uses only a small number of edges from $D - W$ and that $|\mathcal{P}|$ is small. This will be useful later as, in combination with Proposition~\ref{pr:PartDecomp2}, it allows us to remove a small number of problematic vertices from our digraph $D$ at the expense of a small reduction in $\disc(D)$. This is the content of Lemma~\ref{lma:W} below and our goal in this section is to prove it.

\begin{lemma} \label{lma:W}
Let $n \in \mathbb{N}$ and $0 < 1/n \ll \alpha, \beta \ll \gamma \ll  1 $ and $0 < 1/n \ll \eps  \ll 1 $ 
\COMMENT{VP: separated eps from the hierarchy because I need this in my part. Have checked the proofs, but please recheck.
AL: OK, actually, we are fine with $0 < 1/n \ll \alpha, \beta, \eps \ll \gamma \ll  1$}
and $C \ge 32$. \COMMENT{AL:changed inequality on $C$.}
Let $D$ be an oriented graph on $n$ vertices such that $\delta (D) \ge (1 - \eps) n $ and $\disc(D) \ge C n$.
Let $W \subseteq V(D)$ of size $|W| \le \beta n $.
Suppose that $| \disc_D(v) | \le \alpha n$ for all $v \in V(D) \setminus W$.
Then there exists a partial decomposition $\mathcal{P}$ of $D$ such that writing~$H = \cup \mathcal{P}$ we have  
\begin{enumerate}[label={\rm(\roman*)}]
	\item \label{itm:H2} for all $v \in V(D) \setminus W$, $d_H(v) = 2d$ for some $d \le (18 \beta + 4 \gamma) n $;\COMMENT{AL: Change the inequality on $d$, before $d \le (14 \beta + 4 \gamma) n$.}
	\item \label{itm:H3} $H[W] = D[W]$;
	\item \label{itm:H4} for all $w \in W$, if $\disc^{\pm}_D(w) \ge 0$, then $d^{\mp}_{D-H} = 0$;
	\item \label{itm:H5} $\disc(D-H) = \disc(D) - \disc (H) \ge C n/4$.
\end{enumerate}
\end{lemma}

Note that (iii) guarantees that for every $w \in W$ with $\disc(w) \geq 0$ (resp.\ $\disc(w) \leq 0$), every edge of the form $vw$ (resp.\ $wv$) is in $H$ and we informally refer to such edges as going in the ``wrong'' direction.
The poof of Lemma~\ref{lma:W} is split into two lemmas, Lemmas~\ref{lma:D[W]} and~\ref{lma:D[W,V']}. In Lemma~\ref{lma:D[W]}, we deal with all edges inside $W$ and in Lemma~\ref{lma:D[W,V']}, we deal with the edges between $W$ and $V(D) \setminus W$ going in the ``wrong'' direction. 
The basic idea in each case is as follows. Write $F$ for the set of edges incident with $W$ which we wish to remove from $D$ (and thus to add to $H$).
Each of these edges can be thought of as a path and we start by extending these paths (if necessary) so that their endpoints lie in $V(D) \setminus W$ to give a set of paths $\mathcal{Q}$. The reason for doing this is that $D-W$ is a robust expander and so has good connectivity properties; this allows us to connect the large number of paths in $\mathcal{Q}$ into a small number of long paths $\mathcal{Q}'$ (see Corollary~\ref{cor:matchpath}). At the same time we can ensure the paths in $\mathcal{Q}'$ have suitable start and endpoints so that $\mathcal{Q}'$ is a partial decomposition with a small number of paths that contains all edges in $F$. While this is conceptually quite simple, the process of extending the paths into $V(D) \setminus W$ and choosing appropriate start and endpoints becomes technical if we wish to ensure that the paths we create do not interfere with each other.

Before we can prove these two lemmas, we will need a technical definition and one preliminary result.

Consider a digraph $D$ and a vertex subset $W \subseteq V(D)$. Let $V = V(D) \setminus W$. Suppose we have two internally vertex-disjoint paths $P, P'$ that both start at some $x \in V(D)$ and end at some different vertex $y \in V(D)$. Now starting with $P \cup P'$ delete any edges of $P\cup P'$ that occur inside $V$; this is essentially what we refer to as a $(W,V)$-path system, which is formally defined below.
\begin{definition}
Let $W$ and $V$ be disjoint vertex sets and let $X$, $Y$, and $J$ be sets of paths on $W \cup V$. We write for example $V(J)$ to mean the set of all vertices of all paths in $J$.

We say that $ ( X , Y , J ) $ is a \emph{$(W,V)$-path system} if there exist distinct vertices $x$ and $y$ such that
\begin{enumerate} [label={ \rm (P\arabic*)}]
	\item $X = \{ x \}$ if $x \in V$; otherwise $X$ is a set of two edge-disjoint paths that both start at $x$ and  end in $V$;
	\item $Y = \{ y \}$ if $y \in V$; otherwise $Y$ is a set of two edge-disjoint paths that both start in $V$ and end at $y$; 
	\item $J$ is a set of vertex-disjoint paths such that each path in $J$ has both endpoints in $V$;
	\item \label{itm:disjoint} $d_{X \cup Y \cup J} (v)  \le 1$ for all $v \in V$;
	\item $V(X)$, $V(Y)$, and $V(J)$ are disjoint.
\end{enumerate}
We will often take $X = \{ x x',  x x''\}$ for some $x',x'' \in V$ if $x \in W$ and similarly for $Y$. We will interchangeably think of $X$, $Y$, and $J$ both as a set of paths and as the graph which is the union of those paths, but it will always be clear from the context.

We say that the two paths $P_1$ and $P_2$ \emph{extend} $ ( X , Y , J ) $, if $X \cup Y \cup J \subseteq P_1 \cup P_2$ and each $P_i$ starts at~$x$ and ends at~$y$. 
We refer to $x$ and $y$ as the \emph{source} and \emph{sink}, respectively.
\end{definition}

The following corollary (of Lemma~\ref{lma:manyshortpaths3}) shows how to simultaneously extend a collection of vertex-disjoint $(W,V)$-path systems so that the resulting paths are internally vertex-disjoint..

\begin{corollary} \label{cor:matchpath}
Let $n,s \in \mathbb{N}$ and $0 < 1/n \ll \eps, \eps' \ll 1$ and $1/n \ll 1/s$.
	\COMMENT{AL: added condition on $s$. VP: I don't think we need this. It is implicit that $s \leq \eps' n$}
Let $D$ be an oriented graph with vertex partition $V(D) = W \cup V$ such that $|V| = n$ and $\delta^0( D [ V ] ) \geq (1/2 - \eps) n$.
For $i \in [s]$, let $( X_i , Y_i , J_i)$ be $(W,V)$-path systems. 
Suppose that the sets $V_i := V(X_i \cup Y_i \cup J_i) \cap V$ for $i \in [s]$ are disjoint and that $|\bigcup_{i \in [s]}V_i| \le \eps' n$.  
Then $D \cup \bigcup_{i \in [s]} (X_i \cup Y_i \cup J_i)$ contains paths $P_1,P'_1,  \dots, P_s, P_s'$ such that
\begin{enumerate}[label = {\rm (\alph*)}]
	\item for each $i \in [s]$, $P_i$ and $P'_i$ extend $( X_i , Y_i , J_i )$;
	\item $d_{ \bigcup_{i \in [s]} (P_i \cup P_i')} (v) = 2$ for all $v \in V$.
\end{enumerate}
\end{corollary}

\begin{proof}
Let $\tilde{V} = \cup_{i \in [s]}V_i$ so that $s \leq |\tilde{V}| \leq \eps'n$.
Let $\tilde{E}$ be the set of edges used in all the paths in all the path systems $(X_i, Y_i, J_i)$ for all $i \in [s]$.
Write $D' = D[V] - \tilde{E}$.

For each $i \in [s]$ let $P_{i1}, \ldots, P_{it(i)}$ be the paths of~$J_i$.
We will apply Lemma~\ref{lma:manyshortpaths3} to join the paths of our path systems together. 
Let $a_{ij}$ and $b_{ij}$ be starting and ending points of~$P_{ij}$, respectively, so $a_{ij},b_{ij} \in \tilde{V}$.
Also, let $x_i, x_i'$ be the two end-points in $V$ of the paths in $X$ and let $y_i, y_i'$ be the two end-points in $V$ of the paths in $Y$ (where possibly $x_i = x'_i$ and/or $y_i=y'_i$).  

Let $H := \bigcup_{i \in [s]} T_i$ be a multigraph on $\tilde{V} \subseteq V$, where
\begin{align*}
T_i := \{ x_ia_{i1}, b_{i1}a_{i2}, b_{i2}a_{i3}, \ldots, b_{i(t(i)-1)}a_{it(i)}, b_{it(i)}y_i, x_i'y_i' \}.
\end{align*}
By property~\ref{itm:disjoint} of path systems $T_i$ is a matching and since the $V_i$ are disjoint, then $H$ is a matching on $\tilde{V}$ so $|E(H)| \leq |\tilde{V}| \leq \eps'n \le  |D'|$. 
Note that $\delta^0( D' ) \geq  \delta^0( D [ V ] ) - |\tilde{V}| \geq (1/2 - \eps - \eps') n \geq 3n/7$.
We apply Lemma~\ref{lma:manyshortpaths3} with $D',H,\emptyset,2\eps'$ playing the roles of $D, H_1, S_1, \gamma$ and obtain a set of edge-disjoint paths $\mathcal{Q} := \{Q_e \colon e \in E(H)\}$ such that
\begin{itemize}
\item for each $e = xy \in E(H)$, $Q_e$ is a path from $x$ to $y$;
\item the paths $Q_e : e \in E(H)$ are vertex-disjoint (since $H$ is a matching) 
\item $V(\cup \mathcal{Q})= V$ and $\Delta( \cup \mathcal{Q}) \leq 2$.
\end{itemize}
For each $i \in [s]$ set 
\begin{align*}
P_i := (X_i \cup Y_i \cup J_i) \cup  \bigcup_{e \in T_i \setminus \{ x_i'y_i' \} }Q_e 
\:\:\:\:\: \text{and} \:\:\:\:\:
P_i' = Q_{x_i'y_i'}.
\end{align*}
Note that $P_i$ forms a path by our choice of $T_i$ and that $P_i, P_i'$ extends $(X_i, Y_i, J_i)$; thus conditions (a) and (b) of the corollary are satisfied.
\end{proof}

Our first step towards proving Lemma~\ref{lma:W} is Lemma~\ref{lma:D[W]} below where we construct a partial decomposition that uses all the edges inside~$W$.

\begin{lemma} \label{lma:D[W]}
Let $n \in \mathbb{N}$ and $0 < 1/n \ll \alpha, \beta,   \eps \ll 1 $.
Let $C \ge 32$. \COMMENT{AL:changed inequality on $C$.}
Let $D$ be an oriented graph on $n$ vertices such that $\delta (D) \ge (1 - \eps) n $ and $\disc(D) \ge C n$.
Let $W \subseteq V(D)$ of size $|W| \le \beta n $.
Suppose that $| \disc_D(v) | \le \alpha n$ for all $v \in V(D) \setminus W$.
Then there exists a partial decomposition $\mathcal{P}$ of $D$ such that writing~$H = \cup \mathcal{P}$ we have 
\begin{enumerate}[label={\rm(\roman*)}]
	\item \label{itm:W1} $H[W] = D[W]$;
	\item \label{itm:W2} $\Delta(H) \le 21 |W| $ and $d_{H} (v) = 18 |W|$ for all $v \in V(D) \setminus W$; \COMMENT{AL: numbers changed here before 17 and 14.}
	\item \label{itm:W4} $\disc(D-H) \ge Cn/2$.
\end{enumerate}
\end{lemma}

\begin{proof}
Let $\gamma >0$ be such that $\alpha, \beta, \eps \ll \gamma \ll 1 $. 
\COMMENT{AL: changed the hierarchy for $\gamma$}
\COMMENT{VP: I think we also need $\alpha \ll \gamma$}
Let $\ell := |W|$ and let $V' := V(D) \setminus W$. 
Note that
\begin{align}
	\label{eqn:d0(D[V'])}
	\delta^0 ( D[V'] ) \ge (1/2 - \eps - \alpha- \beta)n .
\end{align}
Let $W^{\pm}_{\gamma} := \{ w \in W : \disc^{\pm}(w) \ge ( 1 - \gamma ) n \}$ and $W_0 := W \setminus (W_{\gamma}^+ \cup W_{\gamma}^-)$.
By Vizing's theorem, $D[W]$ can be decomposed into $\ell$ (possibly empty) matchings $M_1, \dots, M_{\ell}$.
For each $ i \in [\ell]$, we partition $M_i$ into matchings
\begin{align*}
M_i^0 &:= \{ ab \in M_i \colon a \in W^+_{\gamma}, b \in W^-_{\gamma}\}, \\ 
M_i^+ &:= \{ ab \in M_i \colon a \in W^+_{\gamma} , b \not\in W^-_{\gamma} \}, 
 \\
M_i^- &:= \{ ab \in M_i   \colon a \not\in W^+_{\gamma}, b \in W^-_{\gamma} \}, \\
M_i' & := \{ ab \in M_i   \colon a \not\in W^+_{\gamma}, b \not\in W^-_{\gamma} \}= M_i \setminus (M_i^0 \cup M_i^+ \cup M_i^-).
\end{align*}
Let $m^{\ast}_i := |M^{\ast}_i|$ for all $\ast \in \{0, \prime, +,-\}$.
Note that for each $i$
\begin{align}
\label{eqn:mi}
	m_i^0 + m_i^+ + m_i^- \le |W^+_{\gamma}|+|W^-_{\gamma}|.
\end{align}

Suppose that we have found partial decompositions $\mathcal{P}_1 , \ldots, \mathcal{P}_{\ell}$ such that writing $H_j = \cup \mathcal{P}_j$, we have 
for each $j \in [ \ell ]$,
\begin{enumerate} [label={\rm (\roman*$'$) }]
	\item \label{itm:Hi5} $\mathcal{P}_j$ is a partial decomposition of $D_{j-1} := D - (H_1 \cup \cdots \cup H_{j-1})$ (and hence $\cup_{j \in [\ell]} \mathcal{P}_j$ is a partial decomposition of $D$ by Proposition~\ref{pr:PartDecomp1}(b)); 
	\item \label{itm:Hi1} $H_j[W] = M_j$;
	\item \label{itm:Hi2} $\Delta(H_j) \le 21$ and $d_{H_j} (v) = 18 $ for all $v \in V'$;
	\item \label{itm:Hi3} $|\mathcal{P}_j| = \disc(H_j) = m_i^0 + 2 m_i^+ + 2 m_i^- + 4$;
	\item \label{itm:Hi6} $H_1, \dots, H_{j}, M_{j+1}, \dots, M_{\ell}$ are edge-disjoint. 
\end{enumerate}
Set $\mathcal{P}:= \cup_{j \in [\ell]}\mathcal{P}_j$ and $H:= \cup_{j \in [\ell]} H_j$. 
Clearly \ref{itm:W1} and \ref{itm:W2} hold. 
To see \ref{itm:W4}, note that \ref{itm:Hi3} and \ref{itm:Hi5} imply that 
\begin{align*}
	\disc(D-H) = \disc(D) - \disc(H)
	\overset{\mathclap{\eqref{eqn:mi}}}{\ge} \disc(D) - 2(|W^+_{\gamma}|+|W^-_{\gamma}| + 2)\ell.
\end{align*}
If $|W^+_{\gamma}|+|W^-_{\gamma}| \le 2C$, then $\disc(D-H) \ge Cn - 4(C  +1) \ell \ge Cn/2 $.
If $|W^+_{\gamma}|+|W^-_{\gamma}| > 2C$, then 
\begin{align*}
	\disc(D-H) 	& \ge \disc(D) - 3(|W^+_{\gamma}|+|W^-_{\gamma}|)\ell \\
	& = \frac12 \sum_{v \in V(D)} \disc_D(v) - 3(|W^+_{\gamma}|+|W^-_{\gamma}|)\ell \\ 
	& \ge  \frac{( 1- \gamma) n}{2} (|W^+_{\gamma}|+|W^-_{\gamma}|) - 3(|W^+_{\gamma}|+|W^-_{\gamma}|)\ell \\
	& \ge ( 1- \gamma- 6 \beta) n (|W^+_{\gamma}|+|W^-_{\gamma}|) /2 \ge Cn/2.
\end{align*}
Therefore to prove the lemma, it suffices to show that such $\mathcal{P}_1, \dots, \mathcal{P}_{\ell}$ exist.

Suppose for some $i \in [\ell]$, we have already found partial decompositions $\mathcal{P}_1, \dots, \mathcal{P}_{i-1}$ satisfying~\ref{itm:Hi5}--\ref{itm:Hi6}.
We now construct $\mathcal{P}_i = \mathcal{P}'_i \cup \mathcal{P}^+_i \cup  \mathcal{P}^-_i \cup \mathcal{P}^0$,
 where $\mathcal{P}_i^*$ is a partial decomposition containing the edges of $M_i^*$ for $* \in \{+,-,',0\}$.
We immediately define $\mathcal{P}_i^0 = M_i^0$.
We will write $H_i, H'_i, H^+_i, H^-_i, H^0_i$ 
 respectively for the union of paths in $\mathcal{P}_i, \mathcal{P}'_i, \mathcal{P}^+_i, \mathcal{P}^-_i, \mathcal{P}_i^0$.
	\COMMENT{AL:added this sentence}
Let $D^0_{i-1} := D_{i-1} - H^0_i - (M_{i+1} \cup \cdots \cup M_{\ell})$.
\COMMENT{VP: changed the definition of $D^0_{i-1}$ by removing the matchings $M_{i+1}, \ldots, M_{\ell}$ in order to satisfy (v').}
		\COMMENT{AL:For the rest of the proof changed $D_i$ to $D^0_{i-1}$.}
Note that by~\eqref{eqn:d0(D[V'])} and~\ref{itm:Hi2},
\COMMENT{VP:need here that $\alpha \ll \gamma$} 
\begin{align}
	\label{eqn:d0(D_i[V'])}
	\delta^0 ( D^0_{i-1} [V'] ) \ge \delta^0 ( D[V'] ) - 21 (i-1) - \ell
	 \ge (1/2 - \gamma) n 
\end{align}
and (by a similar argument as used to bound $\disc(D-H)$) we have
\begin{align}
	\label{eqn:disc(D_i)}
	\disc(D^0_{i-1}) \ge Cn/2.
\end{align}

We first construct the partial decomposition $\mathcal{P}_i'$ of~$D^0_{i-1}$ containing~$M_i'$ in the following claim.
\COMMENT{AL: put the construction of $H_i', H_i^{pm}$ into claims}

\begin{claim} \label{clm:Hi'}
There exists a partial decomposition $\mathcal{P}_i'$  of $D^0_{i-1}$ such that, recalling $H_i' = \cup \mathcal{P}_i'$, we have 
\begin{enumerate}[label={\rm (a$_{\arabic*}$) }]
	\item \label{itm:Hi'0} $|\mathcal{P}_i'| = 4$ (and there exist vertices $x_1,x_2, y_1, y_2$ such that two of the paths start $x_1$ and end at $y_1$ and the other two start at $x_2$ and end at $y_2$);
	\item \label{itm:Hi'} $H_i'[W] = M'_i$, $\Delta(H_i') \le 4$ and $d_{H_i'} (v) = 2$ for all~$v \in V'$.
\end{enumerate}
\end{claim}

\begin{proof}[Proof of Claim]
Let $x_1,x_2,y_1,y_2$ be any four distinct vertices such that $\disc^{+}_{D^0_{i-1}} ( x_j ) \ge 2$ and $\disc^{-}_{D^0_{i-1}} ( y_j ) \ge 2$ for all $j \in [2]$.
Note that such vertices exist by~\eqref{eqn:disc(D_i)}.
Consider $j \in [2]$.
If $x_j \in V'$, then set $X_j = \{x_j\}$; if $x_j \in W$, then $x_j \not\in W_{\gamma}^-$ since $\disc(x_j)>0$.
So $d^+_{D^0_{i-1}} ( x_j , V') \ge \gamma n /2 \ge 2|W|+4$ and we can set $X_j = \{x_j x'_j, x_j x''_j\}$ for some distinct $x'_j, x''_j \in N^+_{D^0_{i-1}} (x_j) \cap V'$. 
Similarly, if $y_j \in V'$, then set $Y_j = \{y_j\}$; if $y_j \in W$, then set $Y_j = \{ y'_j y_j, y''_j y_j\}$ for some distinct $y'_j, y''_j \in N^-_{D^0_{i-1}} (y_j) \cap V'$. 
Moreover, we may further assume that $X_1, X_2, Y_1, Y_2$ are vertex-disjoint.
Let $U:= V(X_1) \cup V(X_2) \cup V(Y_1) \cup V(Y_2)$.
Partition $M_i'$ into $M^1$ and $M^2$ such that (by relabelling $X_1,X_2, Y_1,Y_2$ if necessary) $V(M^j) \cap (X_j \cup Y_j) = \emptyset$ for $j \in [2]$.
Let 
\begin{align*}
	M^1 := \{a_j b_j \colon j \in [r] \} 
	\text{ and }
	M^2 := \{ a_{r+j} b_{r+j} \colon j \in [s] \}.
\end{align*}
For each $j \in [r+s]$, note that $a_j \in W \setminus W^+_\gamma$
	\COMMENT{AL: Before we had $a_j \in W_0$, also change the calculation below. This is where we need $\eps, \beta \ll \gamma$.}
	and so
\begin{align*}
	d^-_{D^0_{i-1}} ( a_j , V') & \overset{\mathclap{\text{\ref{itm:Hi2}}}}{\ge} 
	d^-_{D}(a_j)  - |W| -21|W| 
	= \frac{ d_D(a_j) - \disc_D^+(a_j)}2  - 22|W| \\
	& \ge \frac{(1 - \eps)n - (1 - \gamma)n }2 - 22 \beta n \ge 2 \beta n +4
	\ge 2|W|+4.
\end{align*}
By a similar argument, we have $d^+_{D^0_{i-1}} ( b_j , V') \ge 2|W|+4$.
So there exist distinct $a'_1, \dots, a'_{r+s}, b'_1, \dots, b'_{r+s} \in V' \setminus U$ such that $a'_j \in N^-_{D^0_{i-1}}(a_j)$ and $b'_j \in N^+_{D^0_{i-1}}(b_j)$ for all $j \in [r+s]$. 
Let 
\begin{align*}
J_1 & := \{ a'_j a_j b_j b'_j \colon j \in [r] \}, &
J_2 & := \{ a'_{r+j} a_{r+j} b_{r+j} b'_{r+j} \colon j \in [s] \} .
\end{align*}
Observe that $(X_1, Y_1, J_1)$ and $(X_2,Y_2,J_2)$ are $(W,V')$-path systems. Note further that  
\COMMENT{VP: added extra two lines below since it's a condition of the corollary}
$X_1,Y_1,J_1,X_2,Y_2,J_2$ are vertex-disjoint
and their union has size at most $2|W| + 4 \leq 3\beta n$.
	\COMMENT{AL:before was $|W|+2$}
By considering $(X_1, Y_1, J_1), (X_2,Y_2,J_2)$ and~\eqref{eqn:d0(D_i[V'])}, Corollary~\ref{cor:matchpath} implies that 
$D^0_{i-1} [V'] \cup J_1 \cup J_2$ contains paths $P_1, P'_1,P_2,P'_2$ such that, for $j \in [2]$, $P_j$ and $P_j'$ extends $(X_j, Y_j, J_j)$ and $d_{P_1 \cup P'_1 \cup P_2 \cup P'_2} (v) = 2$ for all $v \in V'$.
Let $\mathcal{P}_i' := \{ P_1, P'_1, P_2, P'_2 \}$.
It is easy to check that $\mathcal{P}_i'$ has the desired properties.
\end{proof}

In the next claim, we construct the partial decompositions $\mathcal{P}_i^+$ and $\mathcal{P}_i^-$ of~$D_{i-1}':= D^0_{i-1} - H'_i$ containing~$M_i^+$ and $M_i^-$ respectively as follows. 
\begin{claim}
There is a partial decomposition $\mathcal{P}_i^+ \cup \mathcal{P}_i^-$ of $D_{i-1}'$ such that, recalling $H_i^{\pm} = \cup \mathcal{P}_i^{\pm}$, we have 
\begin{enumerate}[label={\rm (b$_{\arabic*}$) }]
	\item \label{itm:Hi+1} $|\mathcal{P}_i^{\pm}| = 2 m^{\pm}_i$;  
	\item \label{itm:Hi+2} $H_i^{\pm}[W] = M^{\pm}_i$, $\Delta(H_i^{\pm}) = 8$ and $d_{H_i^{\pm}} (v) = 8$ for all~$v \in V'$.
\end{enumerate}
\end{claim}

\begin{proof}[Proof of Claim]
\COMMENT{AL: Turns out that I need to split into FOUR matchings instead of THREE}
First we arbitrarily partition $M_i^+$ into four matchings, which we denote by $N_1, N_2, N_3,N_4$, each of size $\lfloor m_i^+/4 \rfloor$ or $\lceil m_i^+/4 \rceil$.
Let $m := |N_1|$ and $N_1 = \{ w_j w'_j \colon j \in [m] \}$.

We show that there exist distinct vertices $z_1, \dots, z_{m}  \in V(D) \setminus V(N_1)$ such that, for all $j \in [m]$, $\disc^{-}_{D_{i-1}'} ( z_j ) \ge 2 $.
	\COMMENT{AL:Edited the justification.}
Indeed if $m \le C / 8$ then $\disc (D'_{i-1}) \ge \disc(D^0_{i-1}) - 4 \ge 3Cn/8 \ge 3m n = (|V(N_1)|+m)n$ by~\eqref{eqn:disc(D_i)} so we can find such $z_j$ in this case.
On the other hand, if $m \ge C / 8 \ge 4$, then 
	\COMMENT{AL: The next calculation is where I need 4 matchings. }
\begin{align*}
\disc(D_{i-1}') 
& \ge \sum_{w \in W^+_{\gamma}} \disc^+_{D_{i-1}'} (w) 
\ge | W^+_{\gamma} |  \left( (1- \gamma)n  - \sum_{j \in [i-1]} \Delta(H_j) - \Delta(H_i')-1 \right) \\
& \overset{\mathclap{\text{\ref{itm:Hi2},\ref{itm:Hi'}}}}{\ge}   | M^+_i |  (1-2\gamma)n
\ge (4 m - 3) (1- 2 \gamma) n   \ge 3m n = (|V(N_1)|+m)n,
\end{align*}
so again we can find the desired $z_j$.

By a similar argument as used in the proof of Claim~\ref{clm:Hi'}, there exist $(W,V')$-path systems $(W_j, Z_j, \emptyset)$ for $j \in [m]$ such that, for all $j \in [m]$,
\begin{itemize}
	\item $W_j = \{ w_j w'_jw''_j, w_j w'''_j\}$ with $w''_j,w'''_j \in V'$
	\item if $z_j \in V'$, then $Z_j = \{z_j\}$; otherwise $Z_j = \{z'_j z_j, z''_j z_j\}$ for some $z'_j, z''_j \in V'$;
	\item the $2m$ graphs $W_j$, $Z_j$ with $j \in [m]$ are vertex-disjoint and are subgraphs of $D'_{i-1}$.
\end{itemize}
By considering $(W,V')$-path systems $(W_j, Z_j, \emptyset)$ (and using~\eqref{eqn:d0(D_i[V'])} and \ref{itm:Hi'}), Corollary~\ref{cor:matchpath} implies that 
$D_{i-1}'[V']$ contains paths $P_1, P'_1, \dots, P_{m},P'_{m}$ such that, 
\begin{itemize}
	\item for all $j \in [m]$, $P_j$ and $P_j'$ extend $(W_j, Z_j, \emptyset)$;
	\item $d_{\bigcup_{j \in [m]}(P_j \cup P'_j)} (v) = 2$ for all $v \in V'$.
\end{itemize}
Let $\mathcal{P}^+_{i,1}  :=  \{  P_j,  P'_j:  j \in [m] \}$ and $H^+_{i,1}  := \cup \mathcal{P}^+_{i,1}$.
Note that $|\mathcal{P}_{i,1}| = 2m$ (where two paths start at $w_j$ and end at $z_j$ for every $j \in [m]$).
Moreover, $H^+_{i,1}[W] = N_1$, $\Delta(H^+_{i,1}) = 2$, $d_{H_{i,1}^+} (v) = 2$ for all~$v \in V'$ and $\mathcal{P}_{i,1}$ is a partial decomposition $D_{i-1}'$ by the choice of $z_j, w_j$.
By a similar argument, $D_{i-1}' - H^+_{i,1}$ has edge-disjoint partial decompositions $\mathcal{P}^+_{i,2},\mathcal{P}^+_{i,3},\mathcal{P}^+_{i,4}$ such that $\mathcal{P}^+_i := \bigcup_{k \in [4]}\mathcal{P}^+_{i,k}$ satisfies \ref{itm:Hi+1} and~\ref{itm:Hi+2}.
By a similar argument, we can construct a partial decomposition $\mathcal{P}^-_i$ of $D_i - H_i' - H_i^+$ satisfying \ref{itm:Hi+1} and~\ref{itm:Hi+2}.
\end{proof}

Finally, we let $\mathcal{P}_i = \mathcal{P}'_i \cup \mathcal{P}^+_i \cup  \mathcal{P}^-_i \cup \mathcal{P}^0$.
Our sequential construction of partial decompositions in the digraphs with earlier partial decompositions removed means that \ref{itm:Hi5} holds by Proposition~\ref{pr:PartDecomp1}(b).
Clearly, \ref{itm:Hi1} holds. Also \ref{itm:Hi6} holds by our definition of $D_{i-1}^0$. 
Note that \ref{itm:Hi3} is implied by \ref{itm:Hi'0}, \ref{itm:Hi+1}.
Finally \ref{itm:Hi2} holds by~\ref{itm:Hi'} and \ref{itm:Hi+2}.
This completes the proof of the lemma.
\end{proof}

In the next lemma, we show how to construct a partial decomposition with few paths that uses all those edges incident with $W$ in the ``wrong'' direction; this will help us to isolate the vertices of $W$ in later sections.

\begin{lemma} \label{lma:D[W,V']}
\COMMENT{VP: Changed hierarchy slightly because I need something slightly stronger (I believe the proof does not need to change, but please check).
Originally: Let $n \in \mathbb{N}$ and $0 < 1/n \ll \alpha, \beta \ll  \gamma \ll \eps   \ll 1 $.
AL:OK
}
Let $n \in \mathbb{N}$ and $0 < 1/n \ll \alpha, \beta \ll  \gamma \ll 1$ and $1/n \ll \eps   \ll 1$.
Let $C \ge 5$. 
Let $D$ be an oriented graph on $n$ vertices such that $\delta (D) \ge (1 - \eps) n $ and $\disc(D) \ge C n$.
Let $W \subseteq V(D)$ of size $|W| \le \beta n $ such that $D[W]$ is empty.
Suppose that $| \disc_D(v) | \le \alpha n$ for all $v \in V(D) \setminus W$.
Then there exists a partial decomposition $\mathcal{P}$ of $D$ such that writing~$H = \cup \mathcal{P}$ we have 
\begin{enumerate}[label={\rm(\roman*)}]
	\item \label{itm:WV1} $|\mathcal{P}| \leq 2 ( 2 + 3 \beta) n$ (equivalently $\disc(H) \le 2 ( 2 + 3 \beta) n$);
	\item \label{itm:WV3} if $w \in W$ with $\disc^{\pm}_D(w) \ge 0$, then $N^{\mp}_{D- H}(w) = \emptyset$;
	\item \label{itm:WV4} for all $v \in V(D) \setminus W$, $d_H(v) = 2d$ for some $d \le 4 \gamma n $.
\end{enumerate}
\end{lemma}

\begin{proof}
Let $q \in \mathbb{N}$ be such that $ \alpha, \beta \ll 1/q \ll \gamma $. 
Let $\ell :=  \lceil \gamma n \rceil$ and $p: = \lceil n/q \rceil \ge \alpha n , \beta n$.
Let $V' := V(D) \setminus W$. 
Note that
\begin{align}
	\label{eqn:d0(D[V'])1}
	\delta^0 ( D[V'] )  \ge (1/2 - \eps - \alpha- \beta)n .
\end{align}
Let 
\begin{align*}
	W^{\pm} &:= \{ w \in W : \disc_D^{\pm}(w)  \ge 0 \}, \\ 
  W^{\pm}_{\gamma} & := \{ w \in W : \disc^{\pm}(w) \ge ( 1 - \gamma ) n \},\\
	W^{\pm}_0 & := W^{\pm} \setminus W^{\pm}_{\gamma}.
\end{align*}
Let $D_0 := D[V', W^+] \cup D[W^-, V']$.
We start by showing that if we can find a family~$\mathcal{S}$ of edge-disjoint $(W,V')$-path systems satisfying the following properties, then the lemma holds:
\begin{enumerate}[label={\rm(\roman*$'$)}]
	\item \label{itm:S1}
	$\mathcal{S} = \{(X_{i,j} , Y_{i,j}, J_{i,j}) \colon i \in [p], j \in [q]\} \cup \{(X_{i} , Y_{i}, J_{i}) \colon i \in [p+1,p+3\ell] \}$;
	\item \label{itm:S4}
	$\bigcup_{i \in [p], j \in [q]} J_{i,j} \cup \bigcup_{i' \in [p+1, p+3\ell]} J_{i'}$ contains $D_0$;
	\item \label{itm:S5}
	each $J_{i,j}$ and $J_{i'}$ consists of vertex-disjoint paths of length~$2$ of the form $awb$ for some $a,b \in V'$ and $w \in  W$;
	\item \label{itm:S6} for all $v \in V'$ we have $2f(v) \leq \disc^+(v)$ and $2g(v) \leq \disc^-(v)$ where $f(v)$ (resp.\ $g(v)$) denotes the number of times $v$ appears as a source (resp.\ a sink) in $\mathcal{S}$ (recall the definition of source and sink for a $(W,V')$-path system);
	\item \label{itm:S2}
	for all $i \in [p]$, $\{ V' \cap V( X_{i,j} \cup Y_{i,j} \cup J_{i,j}) \colon j \in [q] \}$ are disjoint and $|V' \cap \bigcup_{j \in [q]} V( X_{i,j} \cup Y_{i,j} \cup J_{i,j})| \le 2 \beta n +4 $;
	\item \label{itm:S3}
	for all $i \in [p+1,p+3\ell]$, $|V' \cap V( X_{i} \cup Y_{i} \cup J_{i})| \le 2 \beta n +4$.
\end{enumerate}
Let $\mathcal{S}_i  := \{(X_{i,j} , Y_{i,j}, J_{i,j}) \colon  j \in [q]\}$ for $i \in [p]$ and $\mathcal{S}_{i'} := \{(X_{i'} , Y_{i'}, J_{i'})\}$ for $i' \in [p+1,p+3\ell]$.
It is easy to verify that by repeated application of Corollary~\ref{cor:matchpath} (once for each $\mathcal{S}_i$), there exists a set $\mathcal{P}$ of edge-disjoint paths  of $D$ with $\mathcal{P} = \mathcal{P}_1 \cup \cdots \cup \mathcal{P}_{p + 3 \ell}$
and $H_i := \cup \mathcal{P}_i$ such that
\begin{enumerate}[label = {\rm (\alph*)}]
	\item for all $i \in [p]$, we have $|\mathcal{P}_i| = 2q$ with $\mathcal{P}_i = \{P_{i,1},P'_{i,1}, \dots, P_{i,q}, P_{i,q}' \}$;
	\item for each $i \in [p]$ and $j \in [q]$, $P_{i,j}$ and $P'_{i,j}$ extend $( X_{i,j} , Y_{i,j} , J_{i,j} )$;
	\item for each $i' \in [p+1, p+3\ell]$, $\mathcal{P}_{i'} = \{ P_{i'}, P'_{i'} \}$ where $P_{i'}$ and $P'_{i'}$  extends $( X_{i'} , Y_{i'} , J_{i'} )$;
	\item for all $i \in [p + 3\ell]$ and all $v \in V'$, $d_{ H_i } (v) = 2$.
\end{enumerate}
\COMMENT{VP: added sentence below}
Now we check that the conclusion of the lemma holds for $\mathcal{P}$ as defined above.
Note that by the choice of sources and sinks for the path systems, i.e.\ \ref{itm:S6}, $\mathcal{P}$ is a partial decomposition of $D$.
Note also that \ref{itm:WV1} holds since $|\mathcal{P}| = \sum_{i \in [p+3 \ell]} |\mathcal{P}_i| = 2pq + 6 \ell \le 2 ( 2 + 3 \beta) n$. Also \ref{itm:WV3} holds by \ref{itm:S4}. Finally \ref{itm:WV4} holds by (d) as $p+3 \ell \le 4 \gamma n $.  
Thus to prove the lemma, it suffices to show that such $\mathcal{S}$ exists.

\COMMENT{AL:added outline}
Here we give a brief outline of the remainder of the proof.
First we will find all sources and sinks that are required. 
We split $D_0$ into 
\[ 
D_1:= D[V', W^+_{\gamma}] \cup D[W^-_{\gamma}, V'] 
\:\:\:\text{ and }\:\:\:
 D_2:=D[V', W^+_{0}] \cup D[W^-_{0}, V'].
\]
The edges of $D_1$ will be covered by $\mathcal{S}_{p+i'}$ for $i' \in [3\ell]$ and the edges of $D_2$ will be covered by $\mathcal{S}_i$ for $i \in [p]$.
\medskip

\noindent
\textbf{Finding sources and sinks.}
First, we define the sources and sinks for the $(W,V')$-path systems. 
Choose a multiset~$X:= \{x_{i,j} \colon i \in [p], j \in[q] \} \cup \{ x_{p+i'} \colon i' \in [3\ell] \}$ of vertices such that $\disc^+_D(v) \ge 2 f(v) $ for all $v \in V(D)$, where $f(v)$ denotes the number of times $v$ appears in~$X$.
Note that such $X$ exists since 
\begin{align*}
	\disc(D) \ge C n \ge 2 ( pq + 3\ell) + n.
\end{align*}
Similarly, choose a multiset $Y:= \{y_{i,j} \colon i \in [p], j \in[q] \} \cup \{ y_{p+i'} \colon i' \in [3\ell] \}$ of vertices such that $\disc^-_D(v) \ge 2 g(v) $ for all $v \in V(D)$, where $g(v)$ denotes the number of times $v$ appears in~$Y$.
Note that for all $v \in V(D)$,
\begin{align}
	f(v) +g(v) \le |\disc_D(v)|/2. \label{eqn:fg}
\end{align}
Since $\ex_D(v) \leq \alpha n$ for all $v \in V'$ and $\alpha n \le p, \ell $, we may assume that, by relabelling if necessary,
\begin{itemize}
	\item for all $i \in [p]$, the multiset $V' \cap \{ x_{i,1}, \dots, x_{i,q}, y_{i,1}, \dots, y_{i,q} \}$ contains no repeated vertices;
	\item for all $i' \in [\ell]$, the multiset 
	$$V' \cap \{ x_{p+3i'-2}, x_{p+3i'-1}, x_{p+3i'}, y_{p+3i'-2}, y_{p+3i'-1}, y_{p+3i}\}$$ contains no repeated vertices.
\end{itemize}
Note that $x_{i,j}$ and $y_{i,j}$ will be the source and sink for $(X_{i,j}, Y_{i,j}, J_{i,j})$ and  $x_{p+i'}$ and $y_{p+i'}$ will be the source and sink for $(X_{p+i'}, Y_{p+i'}, J_{p+i'})$.
For $i \in [p]$, let $f_{i}, g_i : W \rightarrow [0,q]$
\COMMENT{VP: previously $f_i g_i$ were functions into $[n]$, but changed to $[q]$}
 be functions such that $f_i(w)$ (and $g_i(w)$) is the number of $j \in [q]$ satisfying $w = x_{i,j}$ (and $w = y_{i,j}$, respectively). Our choices here guarantee that \ref{itm:S6} holds.
\medskip

\noindent \textbf{Covering edges in $D_1$.}
Consider any $i' \in [3\ell]$. 
If $x_{p+i'} \in V'$, then set $X_{p+i'} = \{x_{p+i'}\}$.
If $x_{p+i'} \in W$, then (by our choice of $x_{p+i'}$) we have $\disc(x_{p+i'}) > 0$ and so $d^+_{D} ( x_{p+i'} , V') = d^+_{D} ( x_{p+i'}) \ge n/4$.
We can set $X_{p+i'} = \{x_{p+i'} x'_{p+i'}, x_{p+i'} x''_{p+i'} \}$ for some distinct $x'_{p+i'}, x''_{p+i'} \in N^+_{D} (x_{p+i'}) \subseteq V'$. 
Similarly, if $y_{p+i'} \in V'$, then set $Y_{p+i'} = \{ y_{p+i'} \}$.
If $y_{p+i'} \in W$, then we can set $Y_{p+i'} = \{y'_{p+i'} y_{p+i'}, y''_{p+i'} y_{p+i'}\}$ for some distinct $y'_{p+i'}, y''_{p+i'} \in N^-_{D} (y_{p+i'}) \subseteq V'$.
Furthermore, we can assume that all $X_{p+i'}, Y_{p+i'}$ are edge-disjoint and
\begin{itemize}
	\item for all $i' \in [\ell]$, $X'_{p+3i'-2}$, $X'_{p+3i'-1}$, $X'_{p+3i'}$, $Y'_{p+3i'-2}$, $Y'_{p+3i'-1}$, $Y'_{ p + 3i'}$ are vertex-disjoint, where $X'_j = X_j \cap V'$ and $Y'_j = Y_j \cap V'$.
\end{itemize}
Let $	\hat{X} := \bigcup_{i' \in [3\ell] } X_{p+i'}$ and $\hat{Y} := \bigcup_{i' \in [3 \ell] } Y_{p+i'}$.
Note that $D_0, \hat{X}, \hat{Y}$ are edge-disjoint and 
\begin{align}
	\label{eqn:D(XY)}
	\Delta(\hat{X} \cup \hat{Y}) \le 6 \ell.
\end{align}
For all $w \in W^+_{\gamma}$, 
\begin{align*}
	d_{D_1}^- (w) = d_{D}^- (w) = \frac{ d_D(w) - \disc^+_D(w) }{2} \le \gamma n /2 \le \ell
\end{align*}
and, similarly, $d_{D_1 }^+ (w) \le \ell $ for all $w \in W^-_{\gamma}$.
Since $|W| \le \beta n \le \ell$, we deduce that $\Delta( D_1) \le \ell$.
By Vizing's theorem, $D_1$ can be decomposed into $\ell$ matchings $M'_1, \dots, M'_{\ell}$.
Consider any $i \in [\ell]$.
Partition $M'_i$ into three matchings $M_{p + 3i-2}, M_{p + 3i-1}, M_{p + 3i}$ such that for each $j \in [3]$,
\begin{align*}
	V( X_{p + 3i-3+j} \cup Y_{p + 3i-3+j} ) \cap V(M_{p + 3i-3+j}) = \emptyset.
\end{align*}
For each $i' \in [3\ell]$ and each vertex $w \in V(M_{p+i'}) \cap W$, if $w \in W^{ \pm }_{ \gamma }$, then by~\eqref{eqn:D(XY)}
\begin{align*}
	d_{D - \hat{X} - \hat{Y}}^{\pm} (w) \ge d_D^{\pm}(w) - 6\ell \ge  n/4.
\end{align*}
Hence, by a simple greedy argument, we can extend each $M_{p+i'}$ (with $i' \in [3\ell]$) into a graph~$J_{p+i'}$ such that 
\begin{itemize}
	\item 
	$J_{p+i'}$ consists of precisely $|M_{p+i'}|$ vertex-disjoint paths of length~$2$ with starting points and endpoints in $V'$ (and midpoint in $W$) and $J_i$ is vertex-disjoint from $X_{p+i'}$ and $Y_{p+i'}$;
	\item the $9 \ell$ different graphs $X_{p+1}, Y_{p+1}, J_{p+1}, \dots, X_{p+ 3\ell}, Y_{p + 3\ell}, J_{p+3\ell}$ are edge-disjoint. 
\end{itemize}
Note that each $(X_i, Y_i, J_i)$ is a $(W,V')$-path system satisfying \ref{itm:S5} and~\ref{itm:S3}.
Let $\mathcal{S}' := \bigcup_{i' \in [3\ell]} (X_{p+i'} \cup Y_{p+i'} \cup J_{p+i'})$ and note that $\mathcal{S}'$ covers all the edges in $D_1$.
\medskip

\noindent \textbf{Covering edges in $D_2$.}
We now construct $(X_{i,j}, Y_{i,j},J_{i,j})$, which cover all the edges in~$D_2$.
Initially, set $X_{i,j} = \{x_{i,j}\}$, $Y_{i,j} = \{y_{i,j}\}$ and let $J_{i,j}$ be empty for all $i \in [p]$ and $j \in [q]$. 
If $x_{i,j} \in W^+_\gamma$, then $d^+_{D - S'} ( x_{i,j}) \ge n/4 $ and we can set $X_{i,j} = \{x_{i,j} x'_{i,j}, x_{i,j} x''_{i,j} \}$ for some distinct $x'_{i,j}, x''_{i,j} \in N^+_{D-\mathcal{S}'} (x_{i,j}) \subseteq V'$. 
Similarly, if $y_{i,j} \in W_{\gamma}^-$, then set $Y_{i,j} = \{y'_{i,j} y_{i,j}, y''_{i,j} y_{i,j}\}$ for some distinct $y'_{i,j}, y''_{i,j} \in N^-_{D-\mathcal{S}'} (y_{i,j}) \subseteq V'$. 
\COMMENT{VP: sentence added below to take care of sources and sinks in $W_0^+ cup W_0^-$.}
(Later, in Claim~\ref{clm:final} we will modify those $X_{i,j}$ (resp.\ $Y_{i,j}$)  for which $x_{i,j} \in W_0^+$ (resp.\ $y_{i,j} \in W_0^-$).\ )
We can furthermore assume that all $X_{i,j}, Y_{i,j}$ are edge-disjoint and, for all $i \in [p]$,
$X'_{i,1}, \dots, X'_{i,q}, Y'_{i,1}, \dots, Y'_{i,q}$ are vertex-disjoint, where $X'_{i,j} : = V' \cap V(X_{i,j})$ and $Y'_{i,j} : = V' \cap V(Y_{i,j})$.

\COMMENT{VP: added extra informal explanation here to help the reader. Some of it is repetition but I think it will help the reader/referee. Please check. VP:I think the original description was slightly wrong so now corrected.}
Instead of constructing $(X_{i,j}, Y_{i,j}, J_{i,j})$ one at a time, we build them up in rounds, in each round simultaneously adding a little extra to every $(X_{i,j}, Y_{i,j}, J_{i,j})$.
Before proving this, we  describe somewhat informally how to construct the $J_{i,j}$.
Let $w_1, \dots, w_{s}$ be an enumeration of $W^+_0 \cup W^-_0$. 
For simplicity, we further assume that none of the $w_i$ is a source or sink, that is, $f(w_i) = 0 =g(w_i)$.
	\COMMENT{AL:Added this further assumption.}
For each $i \in [p]$ and $k \in [s]$, we will construct sets $A_{i,k} \subseteq N_{D-\mathcal{S}'}^-(w_{k})$ and $B_{i,k} \subseteq N_{D-\mathcal{S}'}^+(w_{k})$ of suitable size (with $|A_{i,k}| = |B_{i,k}| \leq q$) such that for each $i$ the sets $A_{i,1}, \ldots, A_{i,s}, B_{i,1}, \ldots, B_{i,s} \subseteq V'$ are disjoint. 
We further guarantee that 
\begin{align}
\label{eqn:cover1}
\bigcup_{i \in [p]}A_{i,k} &= N_{D-\mathcal{S}'}^-(w_k) 
\:\:\text{ if } w_k \in W_0^+ \\
\label{eqn:cover2}
\:\:\:\:\text{ and }\:\:\:\:  
\bigcup_{i \in [p]}B_{i,k} &= N_{D-\mathcal{S}'}^+(w_k) 
\:\:\text{ if } w_k \in W_0^-.
\end{align}
These sets will be built up in rounds using matchings, but assuming we have these sets, for each $i$, we define $F_i$ to be the graph with edges 
\begin{align*}
	E(F_i) := \bigcup_{ k \in [s] } \{ a w_{k}, w_{k}b \colon a \in A_{i, k}, b \in B_{i, k} \}.
\end{align*}
so that $\bigcup_{i \in [p]}F_i \supseteq D_2$ by \eqref{eqn:cover1} and \eqref{eqn:cover2}. 
	\COMMENT{AL: before was $\bigcup_{i \in [p]}F_i \supseteq D_2$. VP: I have put it back as I still think it should be a superset! E.g. $F$ also contains edges from $W_0^+$ to $V'$, which are not edges of $D_2$}
Notice that $F_i$ is the union of vertex-disjoint oriented stars with centers $w_1, \ldots, w_s$, where the star at $w_i$ has an equal number of edges (at most $q$) entering and exiting $w_i$.	
So it is easy to see that each $F_i$ can be decomposed into $J_{i,1}, \ldots, J_{i,q}$ where each $J_{i,j}$ satisfies~\ref{itm:S5}. Let us prove all of this formally noting that the fact that some of the $w_i$ are sources or sinks will mean we will have to be more careful about the sizes of our $A_{i,j}$ and $B_{i,j}$.

Let $h : W^+_0 \cup W^-_0 \rightarrow [n/2]$ be the function such that if $w \in W^{\pm}_0$, then 
\begin{align*}
	h(w) = d_D^{\mp}(w) = d_{D - \mathcal{S}'}^{\mp}(w) \le n/2.
\end{align*}
So $h(w)$ will correspond to the number of $\{J_{i,j} \colon i \in [p], j \in [q]\} $ that will contain~$w$.
Let $h_1, \dots, h_p:  W^+_0 \cup W^-_0 \rightarrow [0, q]$ be functions such that, for each $w \in W^+_0 \cup W^-_0$, $\sum_{i \in [p]} h_i(w) = h(w)$ and 
\begin{equation}
\label{eq:fgh}
f_i(w) + g_{i}(w) + h_i(w) \le q.
\end{equation}
Indeed this is possible by considering $h'_i(w):= q - f_i(w) - g_i(w) \geq 0$ so that 
\[
\sum_{i \in [p]}h'_i(w) \geq pq - f(w) - g(w) \overset{\text{\eqref{eqn:fg}}}{\geq} n - \frac{1}{2}|\disc_D(w)| \geq n/2 \geq h(w)
\]
 and making a suitable choice of $h_i(w) \leq h_i'(w)$.
Here $h_{i}(w)$ will help determine the number of $\{J_{i,j} \colon j \in [q]\} $ that will contain~$w$.

Recall $w_1, \dots, w_{s}$ is an enumeration of $W^+_0 \cup W^-_0$.
For $i \in [p]$, let $X_i := \bigcup_{j \in [q]} X_{i,j}$ and $Y_i := \bigcup_{j \in [q]} Y_{i,j}$.
Suppose for some $k \in [0,s]$, we have already found $\{A_{i,k'},B_{i,k'}\}_{i \in [p], k' \in [k]}$ such that 
\COMMENT{VP: added extra quantifiers below. ``for each ...''}
\begin{enumerate}[label = {\rm (\alph*$'$)}]
	\item \label{itm:AB1} for each $i \in [p]$, $V' \cap V(X_i), V' \cap V(Y_i), A_{i,1}, \dots, A_{i,k}, B_{i, 1} \dots, B_{i, k}$ are disjoint;
	\item \label{itm:AB2} 
for each $i \in [p]$ and $k' \in [k]$,	 $ | A_{i,k'} | = h_i(w_{k'}) + 2 g_i ( w_{k'} )$ and $ |B_{i,k'} | = h_i (w_{k'}) + 2 f_i ( w_{k'} )$;
	\item \label{itm:AB3} for each $k' \in [k]$, $A_{1,k'}, \dots, A_{p,k'} \subseteq N_{D-\mathcal{S}'}^-(w_{k'})$ are disjoint;
	\item \label{itm:AB4} for each $k' \in [k]$, $B_{1,k'}, \dots, B_{p,k'} \subseteq N_{D-\mathcal{S}'}^+(w_{k'})$ are disjoint. 
\end{enumerate}

\begin{claim}
\label{clm:final}
If $k = s$ then we can construct $\mathcal{S}$ satisfying \ref{itm:S1} - \ref{itm:S3} (so completing the proof of the lemma).
\end{claim}
\begin{proof}[Proof of claim]
Define $F_i$ to be the graph with edge set
\begin{align*}
	E(F_i) := \bigcup_{ k' \in [s] } \{ a w_{k'}, w_{k'}b \colon a \in A_{i, k'}, b \in B_{i, k'} \}.
\end{align*}
\COMMENT{VP:Added this sentence because we never mention that all the edges in $D_2$ get covered.}Note that by \ref{itm:AB1}-\ref{itm:AB4} and our choice of $h$,$h_i$ we have that $F_1 \cup \cdots \cup F_p \supseteq D_2$.

By~\ref{itm:AB3} and~\ref{itm:AB4}, we know that $F_i$ can be decomposed into $(W,V')$-path systems $(X_{i,j}, Y_{i,j}, J_{i,j})$ (one for each $j \in [q]$) such that $(X_{i,j}, Y_{i,j}, J_{i,j})$ has source~$x_{i,j}$ and sink~$y_{i,j}$.\COMMENT{VP: details added below}
To see this we colour the edges of $F_i$ with colours from $[q]$ as follows. 
For each $j \in [q]$ if $x_{i,j} \in W_0^+$ assign colour $j$ to any two out-edges $x_{i,j}x_{i,j}'$ and $x_{i,j}x_{i,j}''$ in $F_i$ at $x_{i,j}$ and (re)set $X_{i,j} = \{ x_{i,j}x_{i,j}', x_{i,j}x_{i,j}'' \}$. If $y_{i,j} \in W_0^-$ assign colour $j$ to any two in-edges  $y'_{i,j}y_{i,j}$ and $y''_{i,j}y_{i,j}$ in $F_i$ at $y_{i,j}$ and (re)set $Y_{i,j} = \{ y'_{i,j}y_{i,j}, y''_{i,j}y_{i,j} \}$. Such edges exist by \ref{itm:AB2}. Given $w \in W_0^+ \cup W_0^-$ write $c(w)$ for the colour assigned (if any) to edges at $w$. Let $F_i'$ be the remaining (i.e.\ uncoloured) edges of $F_i$, noting that there are precisely $h_i(w) \leq q - f_i(w) - g_i(w)$ in-edges and the same number of out-edges at $w$ in $F_i'$. For each $w$, pick any set of colours $S_w \subseteq [q] \setminus \{ c(w) \}$ with $|S_w| = h_i(w)$. Assign distinct colours of $S_w$ first to the in-edges of $F_i'$ at $w$ and then to the out-edges of $F_i'$ at $w$. Now writing $F'_{i,j}$ for the edges of $F_i'$ coloured $j$, we take $J_{i,j} = F'_{i,j}$. 
In particular
\begin{equation}
\label{eq:unionD2}
\bigcup_{i \in [p]} \bigcup_{j \in [q]}(X_{i,j} \cup Y_{i,j} \cup J_{i,j}) = \bigcup_{i \in [p]}F_i \supseteq D_2.
\end{equation}

Now taking $\mathcal{S} = \{(X_{i,j} , Y_{i,j}, J_{i,j}) \colon i \in [p], j \in [q]\} \cup \{(X_{i} , Y_{i}, J_{i}) \colon i \in [p+1,p+3\ell] \}$, we see that \ref{itm:S1} - \ref{itm:S3} hold. Indeed \ref{itm:S1} and \ref{itm:S5} hold by construction. \ref{itm:S4} holds because we showed $\mathcal{S}'$ covers all edges in $D_1$ and \eqref{eq:unionD2} shows all edges in $D_2$ are covered. We showed \ref{itm:S6} holds when choosing sources and sinks. The disjointness condition in \ref{itm:S2} and the edge-disjointness of $\mathcal{S}$ hold by construction. The bounds in \ref{itm:S2}  and \ref{itm:S3} hold by \ref{itm:S5}.\COMMENT{VP: made everything very explicit.} 
\end{proof}

Therefore, we may assume that $k \in [0, s-1]$.
We show how to find $\{ A_{i,k+1}\}_{i \in [p]}$; finding  $\{B_{i,k+1} \}_{i \in [p]}$ is similar.  Without loss of generality, assume that $w_{k+1} \in W_0^+$.
We have
\begin{align}
	\label{eqn:h(w_{k+1})}
	h(w_{k+1}) = 
	d^{-}_{D} (w_{k+1}) 
	& \ge 	
	\frac{ d_{D} ( w_{k+1} ) -  \disc_{D}(w_{k+1})  }{2} 
	\ge \gamma n /2. 
\end{align}
For each $i \in [p]$, set 
\begin{align*}
U_i := \left( V' \cap V(X_i \cup Y_i)\right)  \cup \bigcup_{k' \in [k]} \left( A_{i,k'} \cup  B_{i, k'} \right).
\end{align*}
Note that $U_i$ is the set of ``forbidden'' vertices for $A_{i,k+1}$ and $B_{i,k+1}$ (in order to maintain \ref{itm:AB1}, \ref{itm:AB3}, and \ref{itm:AB4}).

Define an auxiliary bipartite graph~$F_A$ with vertex classes $A$ and~$I$ as follows. Let $A \subseteq N^{-}_{D - \mathcal{S}'} (w_{k+1})$ be of size $h(w_{k+1}) + 2 \sum_{i \in [p]}  g_i(w_{k+1})$\COMMENT{VP: used to say ``of size $h(w_{k+1}) + 2 \sum_{i \in [p]} + _i(w_{k+1})$'' but I think it should be $g$. Please check. Also added explanation that follows. Please check}; this is possible since
\begin{align*}
h(w_{k+1}) + 2 \sum_{i \in [p]} g_i(w_{k+1}) = d^-_{D - \mathcal{S}'}(w_{k+1}) + 0 = d^-_{D - \mathcal{S}'}(w_{k+1}).
\end{align*}
(Note that in the case when we try to find $\{B_{i,k+1} \}_{i \in [p]}$ we use a slightly different calculation\footnote{ 
\begin{align*}
h(w_{k+1}) + 2 \sum_{i \in [p]} f_i(w_{k+1})
&= d^{-}_{D - \mathcal{S}'}(w_{k+1}) + 2f(w_{k+1}) - 2 \sum_{i' \in [p+1, p + 3 \ell]} f_{i'}(w_{k+1}) \\
&\leq d^{-}_{D - \mathcal{S}'}(w_{k+1}) + \disc^+_{D}(w_{k+1}) - \disc^+_{\mathcal{S}'}(w_{k+1})  \\
&= d^{-}_{D - \mathcal{S}'}(w_{k+1}) + \disc^+_{D- \mathcal{S}'}(w_{k+1})  \\
&= d^+_{D - \mathcal{S}'}(w_{k+1}).
\end{align*}
}.)
Let $I$ be a multiset consisting of exactly $h_i ( w_{k+1} ) + 2 g_i ( w_{k+1} )$ copies of $i \in [p]$.
Clearly, $|A| = |I|$.
A vertex~$v \in A$ is joined to $i \in I$ in~$F_A$ if and only if $v \notin U_i$. 
Note that, for all $v \in A \subseteq V'$, $v$ is in at most 
\begin{align*}
f(v) + g(v) + d_D(v,W)  \le ( \alpha  / 2 + \beta) n \le \gamma n / 4 q \overset{\eqref{eqn:h(w_{k+1})}}{\le} |I|/2q
\end{align*}
many of the~$U_i$.
Since each $i \in I$ has multiplicity at most $q$, we deduce that
\begin{align*}
	d_{F_A} (v) \ge |I| -  q \cdot  |I|/2q = |I|/2.
\end{align*}
For each $i \in I$, note that 
\begin{align*}
 | U_i | & \le 2q + \sum_{k' \in [s]} \left( 2h_i (w_{k'})+ 2 f_i (w_{k'}) + 2 g_i(w_{k'})\right) 
	\overset{\mathclap{\eqref{eq:fgh}}}{\le}  2q ( 1 + s ) \\
	&\le 2q ( 1+ \beta n ) 
	\le \gamma n /4
	 \overset{\mathclap{\eqref{eqn:h(w_{k+1})}}}{\le} |A|/2
\end{align*}
implying that 
\begin{align*}
	d_{F_A} (i) \ge |A| -  |A|/2 =|A|/2.
\end{align*}
Therefore, $F_A$ contains a perfect matching~$M$ by Hall's Theorem.
For each $i \in [p]$, define $A_{i,k+1} : = \{v \in A : v i \in M \}$.
By a similar argument, there exist $B_{1,k+1}, \dots, B_{p,k+1} \subseteq N^{+}_{D - \mathcal{S}'}(w_{k+1})$ and by construction the sets satisfy \ref{itm:AB1}-\ref{itm:AB4}. Indeed \ref{itm:AB1} holds by the choice of $U_i$, \ref{itm:AB2} holds by the choice of $I$, and \ref{itm:AB3} and \ref{itm:AB4} hold because $M$ is a matching.
This completes the proof of the lemma.
\end{proof}

We now prove Lemma~\ref{lma:W} using Lemma~\ref{lma:D[W]} and Lemma~\ref{lma:D[W,V']}.

\begin{proof}[Proof of Lemma~\ref{lma:W}]
By Lemma~\ref{lma:D[W]}, there exists a partial decomposition $\mathcal{P}_1$ of $D$ such that writing $H_1 = \cup \mathcal{P}_1$ we have  
\begin{enumerate}[label={\rm(\roman*$'$)}]
	\item  $H_1[W] = D[W]$;
	\item  $\Delta(H_1) \le 21 |W| $ and $d_{H_1} (v) = 18|W|$ for all $v \in V(D) \setminus W$;
	\item  $\disc(D-H_1) \ge Cn/2$.
\end{enumerate}
Let $D_1:= D - H_1$.
Note that $\delta(D_1) \ge ( 1- \eps) n - 21|W| \ge (1- \eps - 21 \beta) n $ and $| \disc_{D_1}(v) | \le | \disc_D(v) | \le \alpha n$ for all $v \in V(D) \setminus W$.
By Lemma~\ref{lma:D[W,V']}, there exists a partial decomposition $\mathcal{P}_2$ of $D_1$ such that writing 
$H_2 = \cup \mathcal{P}_2$ we have  
\begin{enumerate}[label={\rm(\roman*$''$)}]
	\item  $|\mathcal{P}_2| = \disc(H_2) \le 2 ( 2 + 3 \beta) n \leq Cn/4$ ;
	\item  if $w \in W$ with $\disc^{\pm}_D(w) \ge 0$, then $N^{\mp}_{D_1- H_2}(w) = \emptyset$;
	\item  for all $v \in V(D) \setminus W$, $d_{H_2}(v) = 2d$ for some $d \le 4 \gamma n $.
\end{enumerate}
The lemma holds by setting $\mathcal{P} = \mathcal{P}_1 \cup \mathcal{P}_2$, which is a partial decomposition of $D$ by Proposition~\ref{pr:PartDecomp1}(b).
\end{proof}

\section{The final deomposition}
\label{se:finaldecomp}



In this section, we prove Theorem~\ref{thm:LPST2}. We prove it in three main steps as discussed in the overview (Section~\ref{se:overview}). We begin with a tournament $T$ that satisfies the hypothesis of Theorem~\ref{thm:LPST2} but assume that it does not have a perfect decomposition. Gradually we show that certain subdigraphs of $T$ with various additional properties also do not have a perfect decomposition. Finally we show that  these additional properties are in fact sufficient to guarantee a perfect decomposition, giving the desired contradiction.


\subsection{Removing vertices with high excess}
The following theorem allows us to remove vertices of high excess from our tournament to leave an almost complete oriented graph $D$ with slightly smaller excess and with the property that a perfect decomposition of $D$ would give a perfect decomposition of $T$.  

\begin{theorem}
\label{thm:RemHighDisc}
Let $1/n \ll \beta \ll \alpha \ll \eps$ with $n$ even and let $C>32$. 
Let $T$ be an $n$-vertex tournament with $\disc(T) \geq Cn$.
Suppose that $T$ does not have a perfect decomposition.
Then there exists a subdigraph $D$ of $T$ with the following properties:
\begin{enumerate}[label={\rm(\roman*)}]
\item $D$ does not have a perfect decomposition; \label{itm:RemHighDisc1}
\item $|D| \geq (1 - \beta)n$ is even;\label{itm:RemHighDisc2}
\item $d_D(v) \geq (1 - \eps) |D|$ for all $v \in V(D)$;\label{itm:RemHighDisc3}
\item $1 \leq |\disc_{D}(v)| \leq 3 \alpha |D| $ for all $v\in V(D)$;\label{itm:RemHighDisc4}
\item $\disc(D)  \geq (C/4 - 5)n$.\label{itm:RemHighDisc5}
\end{enumerate}
\end{theorem}

We will need the following three relatively straightforward results before we can prove Theorem~\ref{thm:RemHighDisc}.
The first proposition says that any almost regular, almost complete oriented graph has an Eulerian subgraph that uses most of the edges at every vertex and whose removal leaves an acyclic subgraph.

\begin{proposition}
\label{pr:LargeEulerSubgraph}
Let $1/n \ll \eps \ll \eps' \ll 1$.
Suppose that $D$ is an $n$-vertex digraph with $\delta^0(D) \geq \frac{1}{2}(1 - \eps) n$. Then there is an Eulerian digraph $D' \subseteq D$ with $\delta^0(D') \geq \frac{1}{2}(1 - \eps') n$ and such that $D - D'$ is acyclic. 
\end{proposition}

\begin{proof}
Note that $|\disc_D(v)| \leq 2 \eps n$ for every $v \in V(D)$.
Let $K^+$ be the multiset of vertices such that each vertex occurs exactly $\disc^+(v)$ times and let $K^-$ be the multiset of vertices such that each vertex occurs exactly $\disc^-(v)$ times.
Thus $|K^+| = |K^-|$ and write $K^+ = \{ k_1^+, \ldots, k_d^+ \}$ and $K^- = \{ k_1^-, \ldots, k_d^- \}$, where $d = \disc(D)$.
Let $H$ be the directed multigraph on $V(D)$ with $E(H) = \{k_i^+k_i^-: i \in [d] \}$.
Note that $\Delta(H) \le  2 \eps n$. 
We apply Corollary~\ref{cor:manyshortpaths3} and obtain a set of edge-disjoint paths $\mathcal{P} = \{P_{e} \colon e \in E(H)\}$ in~$D$ such that $P_{e}$ has the same starting and ending points as~$e$ and $\Delta(\cup \mathcal{P}) \le 4 \sqrt{2\eps} n $. 
By our choice of $K^+, K^-$, we have that $\mathcal{P}$ is a partial decomposition of~$D$ and that $D' := D - \cup\mathcal{P}$ is Eulerian.
Also $\delta^0(D') \geq \delta^0(D) - \Delta(\cup \mathcal{P}) \geq \frac{1}{2}(1 - \eps')n$.
To ensure that $D - D'$ is acyclic, any cycle in $D - D'$ is added to $D'$.
\end{proof}

Given an oriented graph $D$ for which the underlying undirected graph is slightly irregular, the proposition below will be useful in trying to find a small partial decomposition $\mathcal{P}$ of $D$ such that the underlying undirected graph of $D - \cup \mathcal{P}$ is regular. The function $f$ will record the irregularities in the underlying undirected graph of $D$ and the sets $T_1, \ldots, T_{2tm}$ obtained will identify the vertex sets of the paths in $\mathcal{P}$. Some further technical conditions are present that will be useful later.

Recall that, for $U \subseteq X$, we write $I_U: X \rightarrow \{0,1\}$ for the indicator function of~$U$.

\begin{proposition}
\label{pr:balance}
Let $n,t,m \in \mathbb{N}$ with $tm, 2t \le n$\COMMENT{VP:I don't think these condition is necessary but we can leave it in}. 
Let $V$ be a set with $n$ elements. 
Let $f: V \rightarrow [m]$ be a function with $m:= \max_{v \in V}f(v)$.
Suppose $x_1, \dots, x_{2tm}, y_1, \dots, y_{2tm}$ are elements of $V$ (with repetitions) such that $x_i, y_i, x_{tm+i}, y_{tm+i}$ are distinct for each $i \in [tm]$. 
Then we can find a collection of sets $T_1, \dots, T_{2tm} \subseteq V$ such that
\begin{enumerate}[label={\rm(\roman*)}]
	\item for all $v \in V$, $\sum_{i \in [2tm]} I_{T_{i}}(v) =  f(v) + (2t-1)m$;
	\item $|T_i| \ge  (1 - 1/t ) n$ for all $i \in [2tm]$;
	\item $x_i,y_i \in T_i$  for all $i \in [2tm]$. 
\end{enumerate}
\end{proposition}

\begin{proof}
Given any~$U$, take an arbitrary partition of $V \setminus U$ into sets $A_1, \ldots, A_t$ with $|A_i| \leq n/t$ for all $i \in [t]$ (we allow empty sets in the partition).
Then writing $B_i := V \setminus A_i$, set $\mathcal{S}_U := \{ B_1, \ldots, B_t \}$.
Note that for each $v \in V$, 
\begin{align*}
\sum_{S \in \mathcal{S}_U} I_S(v) = I_U(v) + t-1.
\end{align*}
Since $f(v) \leq m$ for all $v \in V$, we can find sets $U_1, \ldots, U_m$ such that 
$f \equiv  I_{U_1} + \cdots + I_{U_m}$.
Taking $\mathcal{S} = \bigcup_{i\in [m]} \mathcal{S}_{U_i}$, we have $|\mathcal{S}|  = t m$ and
\begin{align*}
\sum_{S \in \mathcal{S}} I_S(v) = 
\sum_{i \in [m]} \left( I_{U_i} + t - 1 \right) 
= f(v) + (t-1)m.
\end{align*}
Write $S_1, \ldots S_{tm}$ for the sets in $\mathcal{S}$.
For $i \in [tm]$, let
\begin{align*}
	T_i := S_i \cup \{x_i,y_i\}
	\text{ and }
	T_{tm+i} := V \setminus (\{x_i,y_i\} \setminus S_i).
\end{align*}
Let $\mathcal{T} := \{ T_i \colon i \in [2tm] \}$. 
Note $|T_i| \geq (1 - 1/t)n$ and $x_i,y_i \in T_i$ for all $i \in [2tm]$. 
For all $v \in V$,
\begin{align*}
\sum_{i \in [2tm]} I_{T_{i}}(v) & =  \sum_{i \in [tm]} ( I_{T_{i}}(v) +  I_{T_{tm+i}}(v)) 
\\
& = \sum_{i \in [tm]} ( I_{S_i} (v) +1 ) = 
f(v) + (2t-1)m.
\end{align*}
\end{proof}

The following Lemma shows how to decompose any almost complete Eulerian oriented graph into a small number of cycles. Some extra technical conditions are placed on the cycles that will be useful later.

\begin{lemma}
\label{lma:EulerDecomp}
Let $n \in \mathbb{N}$ with $1/n \ll \varepsilon \ll 1$.
Suppose $D$ is an $n$-vertex Eulerian oriented graph with $\delta^0(D) \geq \frac{1}{2}(1- \varepsilon)n$. 
Suppose $\phi: V(D) \rightarrow [n]$ satisfies $\sum_{v \in V(D)}\phi(v) \geq 4n$.
Then we can decompose $D$ into $t \leq n$ cycles $C_1, \ldots C_{t}$ where each cycle is assigned two distinct representatives $x_i,y_i \in V(C_i)$ such that no vertex $v \in V(D)$ occurs as a representative more than $\phi(v)$ times.
\end{lemma}

\begin{proof}
We assume $\frac{1}{2}(1- \varepsilon)n$ is an integer.
For $x \in V(D)$, write $f(x) =\frac{1}{2} ( d_D(x) - (1 - \varepsilon)n ) \geq 0$.
Let $t = \lceil \varepsilon^{-1/2} \rceil$ and $m= \max_{x \in V(D)} f(x)$, so $m \leq \varepsilon n$ and $tm \le 2 \sqrt{\eps} n \leq n$.

Let $M$ be the multiset of vertices in which $v \in V(D)$ occurs $\phi(v)$ times so that $|M| \geq 4 n$ and no vertex occurs more than $n$ times.
Let $m_1, m_2 \ldots$ be an ordering of the elements of~$M$ (with multiplicity) from most frequent to least frequent.
For each $i \in [tm]$, write $(x_i,y_i,x_{tm+i}, y_{tm+i}) = (m_i,m_{n+i},m_{2n+i},m_{3n+i})$.
Note that, as vertices,  $x_i,y_i,x_{tm+i}, y_{tm+i}$ are distinct (because no vertex $v$ occurs more than $n$ times in $M$).

By Proposition~\ref{pr:balance}, we can find sets $T_1, \ldots, T_{2tm} \subseteq V(D)$ and vertices $x_1, \dots, x_{2tm},y_1, \dots, y_{2tm} \in V(D)$ such that
\begin{enumerate}[label={\rm(\roman*$'$)}]
	\item for all $v \in V$, $\sum_{i \in [2tm]} I_{T_{i}}(v) =  f(v) + (2t-1)m$;
	\item $|T_i| \ge (1 - 1/t ) n \ge (1- \sqrt{\eps})n$ for all $i \in [2tm]$;
	\item each $T_i$ is assigned two distinct representatives $x_i,y_i \in T_i$;
	\item no vertex $v \in V(D)$ occurs as a representative more than $\phi(v)$ times.
\end{enumerate}

For $i \in[2tm]$, let $S_i := V(D) \setminus T_i$ and $H_i$ be the multidigraph on $V(D)$ with $E(H_i) = \{ x_iy_i, y_ix_i\}$.
Let $H = \bigcup_{i \in [2tm]}H_i$.
Note that $\Delta(H) \le 4tm \le 8 \sqrt{\eps} n$ and $|S_i| \le \sqrt{\eps} n $. 
Apply Lemma~\ref{lma:manyshortpaths3} with $(D, H_i, S_i, 4\sqrt{\eps})$ playing the role of $(D, H_i, S_i, \gamma)$ to obtain edge-disjoint cycles $C_1, \ldots, C_{2tm}$ such that $V(C_i) = T_i$ for each~$i$.


Now, by our choice of $\mathcal{T}$ we have that $C := C_1 \cup \cdots \cup C_{2tm}$ 
satisfies $d_C(x) = 2f(x) + 2(2t-1)m$ and so $D-C$ is a regular Eulerian digraph with $\delta(D-C) \geq (1 - \varepsilon)n - 4tm \geq   3n/7$.
By Lemma~\ref{lma:outexpander} and Theorem~\ref{thm:kelly}, $D-C$ can be decomposed into $s \leq n/2$ Hamilton cycles.
Each of these cycles is assigned two distinct representatives from $M' = M \setminus \{x_1, \ldots, x_{2tm}, y_1, \ldots, y_{2tm} \}$ arbitrarily (this is possible since $|M'| \geq 2 n$ and no vertex occurs more than $n$ times in $M'$).
Thus altogether we obtain a decomposition of $D$ into $t \leq n/2 + 2 tm \leq n$ cycles with representatives as desired. 
\end{proof}

We now prove Theorem~\ref{thm:RemHighDisc}.

\begin{proof}[Proof of Theorem~\ref{thm:RemHighDisc}]
Fix parameters $\eps_0, \eps_2, \eps_2', \eps_3$ such that $ \beta \ll \alpha \ll \eps_0 \ll \eps_2 \ll \eps_2' \ll \eps_3 \ll \eps$.
Let 
\begin{align*}
W^{\pm} := \{ v \in V(T) : \disc^{\pm}(w) \ge \alpha n \},
\text{ }
W := W^+ \cup W^- 
\text{ and }
\overline{W}:= V(T) \setminus W.
\end{align*}
We further guarantee $|W|$ and hence $|\overline{W}|$ is even by moving an arbitrary vertex $v \in \overline{W}$ to $W$ if $|W|$ is odd; in this case $v$ is added to $W^+$ if $\disc(v)>0$ and to $W^-$ if $\disc(v)<0$.
Since $T$ does not have a perfect decomposition, Theorem~\ref{lma:Exactdec-highdiscexp(new)} implies that $\disc(T) < n^{19/10}$.
In particular, 
\begin{align*}
|W| \leq 1 + 2\disc(T)/\alpha n  \leq \beta n.
\end{align*}
So we can apply Lemma~\ref{lma:W} where $(\alpha, \beta, \eps_0/10, \eps_0/10 , C)$ play the role of $(\alpha, \beta, \gamma, \eps, C)$ to obtain a partial decomposition $\mathcal{P}_0$ of~$T$ such that, writing $D_0 : = T - \cup  \mathcal{P}_0$, we have
\begin{enumerate}[label={\rm(a$_{\arabic*}$)}]
	\item $D_0$ does not a perfect decomposition; \label{itm:RDHa1}
	\item $d_{D_0}(v) = d$ for all $v \in \overline{W}$and some odd $d \geq (1 - \eps_0)n$; \label{itm:RDHa2}
	\item $E(D_0[W]) = E_{D_0}(\overline{W},W^+) = E_{D_0}(W^-, \overline{W}) =\emptyset$; \label{itm:RDHa3}
	\item $\disc(D_0) \geq Cn/4$; \label{itm:RDHa4}
	\item $|\disc_{D_0}(v)| \leq \alpha n$ for all $v \in \overline{W}$. \label{itm:RDHa5}
\end{enumerate}  
Since $T$ does not have a perfect decomposition, \ref{itm:RDHa1} holds.
Note that \ref{itm:RDHa2}, \ref{itm:RDHa3}, \ref{itm:RDHa4} follow from Lemma~\ref{lma:W}\ref{itm:H2}, \ref{itm:H3} and~\ref{itm:H4}, and~\ref{itm:H5}, respectively. 
Finally, \ref{itm:RDHa5} follows by our choice of $W$ and the fact that $\mathcal{P}$ is a partial decomposition of~$T$.

Let $\mathcal{P}$ be a partial decomposition of~$D_0$ such that every path in~$\mathcal{P}$ is of the form $w^+ v$, $v w^-$, or $w^+ v w^-$ for some $w^+ \in W^+$, $w^- \in W^-$, $v \in \overline{W}$.
We further assume that firstly the number of paths in~$\mathcal{P}$ of type~$w^+ v w^-$ is maximal and, subject to this, that $\mathcal{P}$ has maximal size.%

Let 
\begin{align*}
D_1 & := D_0 - \cup \mathcal{P}, & 
D_2 &:= D_1- W = D_0- W.
\end{align*}
Note that
\begin{enumerate}[label={\rm(b$_{\arabic*}$)}]
\item $\delta(D_2) \geq d - |W| \geq (1 - \eps_0 - \beta)n \geq (1 - \eps_2)n$;  \label{itm:RDHb1}
\item for every $v \in \overline{W}$, 
$ |\disc_{D_2}(v)| \leq |\disc_{D_1}(v)| + |W| \leq 2 \alpha n$; \label{itm:RDHb2}
\item $\delta^0(D_2) 
\geq \frac{1}{2}[\delta(D_2) - \max_{v}|\disc_{D_2}(v)|]
\geq \frac{1}{2}(1 - \eps_2')n$; \label{itm:RDHb3}
\end{enumerate}

\begin{claim} \label{clm:RDH}
$|\mathcal{P}| < 4n$.
\end{claim}

\begin{proof}[Proof of claim]
Suppose the contrary that $|\mathcal{P}| \geq 4n$.
By Proposition~\ref{pr:LargeEulerSubgraph}, we can find a Eulerian subgraph~$D_3$ of $D_2$ such that $\delta^0(D_3)\geq \frac{1}{2}(1 - \eps_3) n$ and $D_2 - D_3$ is acyclic. 
Let $R : = D_1 - D_3$. 
By~\ref{itm:RDHa3}, any cycle in $R$ lies in $R [W] = D_2 - D_3$.
Hence $R$ is acyclic.
By Proposition~\ref{pr:acyclic}, $R$ has a perfect decomposition~$\mathcal{P}_1$, which is a partial decomposition of $D_0$ by Proposition~\ref{pr:PartDecomp1}(d) and~(b).

We now show that $D_0 - R =  \cup \mathcal{P} \cup D_3$ has a perfect decomposition $\mathcal{P}'$, which will contradict~\ref{itm:RDHa1} (since then $\mathcal{P}_1 \cup \mathcal{P}'$ is partial decomposition of $D_0$ by Proposition~\ref{pr:PartDecomp1}(b)).
Note that each path in $\mathcal{P}$ has a unique vertex in~$\overline{W}$.
For each $v \in \overline{W}$, write $\phi(v)$ for the number of paths in~$\mathcal{P}$ that contain $v$.
Then $\sum_{v \in \overline{W}}\phi(v) = |\mathcal{P}| \geq 4n$.

By Lemma~\ref{lma:EulerDecomp} (with $\eps_3$ playing the role of $\eps$), we can decompose $D_3$ into $t \leq n$ cycles $C_1', \ldots, C_t'$ such that each cycle is assigned two distinct representative vertices $x_i, y_i \in C_i$ such that each vertex $v$ occurs as a representative at most $\phi(v)$ times.
In particular, we can assign two distinct paths $P_i, Q_i \in \mathcal{P}$ to~$C_i$ such that $V(P_i) \cap V(C_i) = x_i$ and $V(Q_i) \cap V(C_i) = y_i$ and $P_1, \ldots, P_t, Q_1, \ldots, Q_t$ are distinct paths of $\mathcal{P}$. 
Now construct $\mathcal{P}'$ from $\mathcal{P}$ by replacing for each $i=1, \ldots, t$ the paths $P_i$ and $Q_i$ by the paths $P_ix_iC_iy_iQ_i$ and $Q_iy_iC_ix_iP_i$.
Now we see $|\mathcal{P}'| = |\mathcal{P}|$ and that the paths in $\mathcal{P}'$ have the same start and endpoints as those in $\mathcal{P}$ so that $\mathcal{P}'$ is a partial decomposition of~$D_0$ by Proposition~\ref{pr:PartDecomp1}(c).
Finally, by construction
\begin{align*}
\cup \mathcal{P}' = 
\cup \mathcal{P} \cup C_1' \cup \cdots \cup C_t' =  
\cup \mathcal{P} \cup D_3 = D_0 - R,
\end{align*}
as required.
\end{proof}

It turns out that if $\disc_{D_2}(v) \ne 0$ for all $v \in \overline{W}$, then one can relatively easily prove the theorem by taking $D = D_2$. However, in order to fulfil condition (iv), we must deal with vertices for which $\disc_{D_2}(v) = 0$: this is not hard but is technically cumbersome. We will modify $\mathcal{P}$ by extending some of its paths. 
Let 
\begin{align*}
U^{\pm} & : =  \{v \in \overline{W} \colon \disc_{D_2}^{\pm}(v)>0\},&
U^0 &: = \overline{W} \setminus (U^+ \cup U^-), \\
U^0_{\pm} & := U^0 \cap \{v \in \overline{W} \colon \disc_{D_0}^{\mp}(v)>0\}.
\end{align*}
Note that $U^0_+$ and $U^0_-$ partition $U^0$ (since $\disc_{D_0}(u) \not= 0$ by~\ref{itm:RDHa2}).
For each $u \in U^0_+$ (and $u \in U^0_-$), let $P_u \in \mathcal{P}$ be a path ending (and starting) at $u$ (such a path exists since $\disc_{D_0}(u) \not= 0$ by~\ref{itm:RDHa2}).
Let $\mathcal{P}'_{\pm} := \{P_u: u \in U^0_{\pm}\} \subseteq \mathcal{P}$ and let $\mathcal{P}' := \mathcal{P}'_- \cup \mathcal{P}'_+$.
\COMMENT{To be precise $\mathcal{P}'_+ = \{ u w^-_u  \colon u \in U^0_+\}$ and $\mathcal{P}'_-= \{ w^+_u u  \colon u \in U^0_-\}$ for some $w^\pm_u \in W^\pm$.}
Our aim is to extend each path in $\mathcal{P}'$ so that its starting and ending points avoid~$U^0$.

We show for later that $\disc(D_2)$ is large.
By the maximality of~$\mathcal{P}$, we have
\begin{align}
N^{\pm}_{D_1}(W^\pm) \subseteq U^{\pm} \cup U^0_{\pm}. \label{eq:neigh}
\end{align}
\COMMENT{
If there exists $w^+ \in W^+$ and $u \in N^{+}_{D_1}(w) \cup U^0_-$, then $ wP_u$ is a path of the form $w^+uw^-$ and so $\mathcal{P} \cup wP_u \setminus P_u$ contradicts the maximality of $\mathcal{P}$. 
If there exists $w^+ \in W^+$ and $u \in N^{+}_{D_1}(w) \cup U^-$, then $\mathcal{P} \cup w u $ contradicts the maximality of $\mathcal{P}$.}
Together with Proposition~\ref{pr:PartDecomp2}(c), we have
\begin{align}
\disc(D_2) 
& = \disc(D_1[\overline{W}]) 
= \disc(D_1) 
= \disc(D_0) - |\mathcal{P}_0| \:\:\:\:
 \overset{\mathclap{\text{\ref{itm:RDHa5}, Claim~\ref{clm:RDH}}}}{>} \:\:\:\:
 (Cn/4) -4n > 4n. 
\label{eq:discD2}
\end{align}

Our aim is to extend each path in $\mathcal{P}'$ so that its starting and ending points avoid~$U^0$. 
In fact, we replace $\mathcal{P}'$ by $\mathcal{Q}'$ using the following claim.

\begin{claim} \label{clm:Q}
There exists a partial decomposition $\mathcal{Q}'$ of $\cup \mathcal{P}' \cup D_1 = D_0 - \cup (\mathcal{P} \setminus \mathcal{P}')$ such that
\begin{enumerate}[label={\rm(c$_{\arabic*}$)}]
	\item $\disc(\cup  \mathcal{Q}'- W ) \le |U^0| \le n $;  \label{itm:RDHd1}
	\item $\cup \mathcal{P}' \subseteq \cup \mathcal{Q}'$; \label{itm:RDHd2}
	\item $\Delta(\cup \mathcal{Q}'- W) \le 2 \eps_3 n $; \label{itm:RDHd3}
	\item \label{itm:RDHd4}
	$1 \le \disc_{D_2 - \cup \mathcal{Q}'}^{\pm}(u) \le 2 \alpha n$ if  $w \in U^0_{\mp}
	\cup U^{\pm}$.
\end{enumerate}
\end{claim}

\begin{proof}[Proof of claim]
We will show how to extend the paths in $\mathcal{P}'_{\pm}$ to obtain sets of paths $\mathcal{Q}_{\pm}$ and we will take $\mathcal{Q} = \mathcal{Q}_+ \cup \mathcal{Q}_-$. We show how to construct $\mathcal{Q}_+$; the construction of $\mathcal{Q}_-$ follows similarly.

For each $u \in U^0_+$, pick a vertex $b_u \in U^-$ such that no $v \in U^-$ is chosen more than $\disc_{D_2}(v) - 1$ times (which is possible as $|U^0_+| \le n \le \disc(D_2) - n$ by \eqref{eq:discD2}) and let $e_u = u b_u$. 
Define a digraph~$H$ on~$V(D)$ with edge set $\{e_u \colon u \in U^0_+\}$.
Note that $\Delta(H) \leq 2 \alpha  n $ by~\ref{itm:RDHb2}.
We apply Corollary~\ref{cor:manyshortpaths3} with $D_2, H, 2 \alpha$ playing the roles of $D, H, \gamma$ to obtain a set of edge-disjoint paths $\mathcal{P}''_+ : = \{P_u' \colon u \in U^0_+\}$ in $D_2$ such that each $P_u'$ starts at~$u$ and ends at~$b_u$ and $\Delta(\cup \mathcal{P}') \le \eps_3 n $.
Recalling that for $u \in U_0^+$, the path $P_u$ is a single edge starting at $W^+$ and ending at $u$, we see that the path $P_u P_u'$ starts at $W^+$ and ends at~$b_u$. 
Let $D_1^+ := \cup \mathcal{P}'_+ \cup D_1 $.
By our choices of $\mathcal{P}'_+$, $b_u$ and~\eqref{eq:neigh}, $\mathcal{Q}_+:=\{ P_u P_u' \colon u \in U^0_+\}$ is a partial decomposition of $D_1^+ - W^-$.
Moreover,  we have 
\begin{align*}
	\disc_{D_2 - \cup \mathcal{Q}_+}(w) 
	\begin{cases}
	\in [\min\{ \disc_{D_2}(w),-1\}, -1]  \overset{\text{\ref{itm:RDHb2}}}{\subseteq} [-2 \alpha n , -1] & \text{if  $w \in U^0_+ \cup U^-$;}\\
	= \disc_{D_2}(w)  & \text{if  $w \in U^0_- \cup U^+$},
	\end{cases}
\end{align*}
where the first case follows since  $\disc_{\cup \mathcal{Q}_+}(u) =\disc_{\cup \mathcal{P}_+'}(u)=1$ for all $u \in U_+^0$, and by our choice of $b_u \in U^-$.
By~\ref{itm:RDHa3} and Proposition~\ref{pr:PartDecomp2}(a) (with $(D_1^+ , \emptyset, W^-,V(D) \setminus W^-)$ playing the role of $(D, A^+,A^-,R)$), we can extend $\mathcal{Q}_+$ to a partial decomposition $\mathcal{Q}'_+ = \{Q_u' \colon u \in U^0_+ \}$ of~$D_1^+$ such that for all $u \in U^0_+$ we have
\begin{enumerate}[label={\rm(d$_{\arabic*}$)}]
	\item $Q_u'-W^- = P_u P_u'$; 
	\item $Q_u'$ is a path from $W^+$ to $U^- \cup W^-$;
	\item $Q_u' - Q_u'[V(D) \setminus W^+] = P_u$;
	\item $\Delta(\cup \mathcal{Q}'_+ - W) \le \eps_3 n $;
	\item for all $w \in \overline{W}$, 
	\begin{align*}
	\disc_{D_2 - \cup \mathcal{Q}'_+}(w) 
	\begin{cases}
	\in [- 2 \alpha n,-1] & \text{if  $w \in U^0_+ \cup U^-$;}\\
	= \disc_{D_2}(w)  & \text{if  $w \in U^0_- \cup U^+$}.
	\end{cases}
\end{align*}.
\end{enumerate}
By a similar argument, we can find a corresponding partial decomposition $\mathcal{Q}'_- = \{Q_u' \colon u \in U^0_- \}$ of~$\cup \mathcal{P}'_- \cup D_1$ edge disjoint from $\cup \mathcal{Q}'_+ $.
By setting $\mathcal{Q}':= \mathcal{Q}'_+ \cup \mathcal{Q}'_-$, our claim follows. Note that \ref{itm:RDHd2}, \ref{itm:RDHd3}, and \ref{itm:RDHd4} follow from (d3), (d4), (d5) respectively, while \ref{itm:RDHd1} follows from (d1) and the fact that $|\mathcal{Q}'_{\pm}| = |U^0_{\pm}|$.
\COMMENT{VP: we didn't say why $\mathcal{Q'}$ is a partial decomp}
\end{proof}

Let 
\begin{align*}
	D_3 & := \cup \mathcal{P}' \cup D_1 - \cup \mathcal{Q}' = D_0 - \cup (\mathcal{P} \setminus \mathcal{P}') - \cup \mathcal{Q}',\\
	D & : = D_3-W = D_2 - \cup \mathcal{Q}'.
\end{align*}
We show that $D$ satisfies the conclusion of the theorem.
In order to prove (i), if $D$ has a perfect decomposition, then so does $D_3$ by \ref{itm:RDHa3} and Proposition~\ref{pr:PartDecomp2}(b), and hence so does $D_0$ since $(\mathcal{P} \setminus \mathcal{P}') \cup \mathcal{Q}'$ is partial decomposition of $D_0$. This contradicts \ref{itm:RDHa1}, so $D$ has no perfect decomposition and so (i) holds.
Our choice of $W$ implies \ref{itm:RemHighDisc2}.
Note that \ref{itm:RemHighDisc3} follows from \ref{itm:RDHb1} and \ref{itm:RDHd3}, and \ref{itm:RemHighDisc4} follows from~\ref{itm:RDHd4}.
Finally to see~\ref{itm:RemHighDisc5},
\begin{align*}
		\disc(D) \geq \disc(D_2) - \disc ( \cup \mathcal{Q}' -W)
		\overset{\mathclap{\text{\eqref{eq:discD2},\ref{itm:RDHd1} }}}{\ge}
		Cn/4 - 5n
\end{align*}
as required.
%
\end{proof}


\subsection{Balancing the number of positive and negative excess vertices}
Given the oriented graph $D$ produced by Theorem~\ref{thm:RemHighDisc}, the following theorem produces a digraph $D'$ that has the same properties as $D$ (with slightly weaker parameters) but with the additional property that the number of vertices of positive excess is almost the same as the number of vertices with negative excess.
Recall that for a digraph $D$,  $U^+(D)$ (resp.\ $U^-(D)$) denotes the set of vertices of $D$ with positive (resp.\ negative) excess.

\begin{theorem}
\label{thm:DiscBalance}
Let $1/n \ll 1/C \ll \alpha, \beta \ll \eps \ll \lambda, \eps' \ll 1$ with $n$ even. 
Suppose that $D$ is an $n$-vertex oriented graph, where $\disc(D) \geq Cn$,  $\delta(D) \geq (1 - \eps) n$, and $1 \leq |\disc_{D}(v)| \leq \alpha n$ for all $v\in V(D)$.
Suppose that $D$ does not have a perfect decomposition. 
Then there exists a subgraph $D'$ of~$D$ with the following properties:
\begin{enumerate}[label={\rm(\roman*)}]
\item $D'$ does not have a perfect decomposition; \label{itm:DB1}
\item $|D'| \geq (1 - \beta)n$ with $|D'|$ even;  \label{itm:DB2}
\item $\delta(D') \geq (1 - \eps'/2)n \geq (1 - \eps') |D'|$; \label{itm:DB3}
\item $1 \leq |\disc_{D'}(v)| \leq \alpha n \leq 2\alpha |D'|$ for all $v\in V(D')$; \label{itm:DB4}
\item $\disc(D') \geq \lambda C n /32 \geq \lambda C|D'| / 32$; \label{itm:DB5}
\item $\left| |U^-(D')| - |U^+(D')| \right| \leq 2\lambda |D'|$.\label{itm:DB6}
\end{enumerate}
\end{theorem}

\begin{proof}
We introduce a parameter $\eps_1$ satisfying $\eps \ll \eps_1  \ll \eps' \ll 1$.
Let us write $U^{\pm} := U^{\pm}(D)$.  
If $\left| |U^-| - |U^+| \right| \leq \lambda n$ then we can take $D' = D$ and we are done, so assume without loss of generality that $|U^-| > |U^+| + \lambda n$.
We make the following claim.

\begin{claim}
There exists sets $X \subseteq U^+$ and $Z \subseteq U^-$ satisfying the following: 
\begin{enumerate}[label={\rm(a$_{\arabic*}$)}]
\item $|X| \leq \beta n$ and $|X|$ is even; \label{itm:DBa1}
\item $Z$ can be partitioned into sets $Z_x \colon x\in X$ with $|Z_x| \leq \disc_D(x)$ and $Z_x \subseteq N_D^+(x)$; \label{itm:DBa2}
\item $n < \sum_{x \in X}\disc_D(x) \leq (1+ 2\alpha) n \leq (1 - \lambda / 4)\disc(D)$ and $\sum_{z \in Z}\disc_D^-(z) \leq (1 - \lambda/4) \disc(D)$; \label{itm:DBa3}
\item $|Z| = \frac{1}{2} \left| |U^-| - |U^+| \right| \pm \frac{1}{2} \lambda n$ or equivalently $\left| |U^- \setminus Z| - |U^+ \cup Z| \right| \leq \lambda n$. \label{itm:DBa4}
\end{enumerate}
\end{claim}

\begin{proof}[Proof of claim]
Assume $\beta n$ is an even integer and let $X'$ be the set of $\beta n$ vertices of $U^+$ of highest excess.
Then
\begin{align*}
\sum_{x \in X'}\disc_D(x) \geq  \beta n \disc(D) / |U^+| > C\beta n^2 / (n/2) = 2C \beta n >n. 
\end{align*}
Now we remove suitable vertices from $X'$ to obtain a set $X$ such that 
\begin{align*}
n < \sum_{x \in X}\disc_D(x) \le (1+ 2\alpha) n \:\:\:\text{ and }\:\:\: |X| \text{ even.}
\end{align*}
This is possible by \ref{itm:DB4}.
For each $x \in X$, we have 
\begin{align*}
	|N_D^+(x) \cap U^-| 
	& \geq \frac{1}{2}d^+_D(x) - |U^+| 
	\geq \frac{1}{2}(1 - \eps)n - |U^+| \\
	&= \frac{1}{2}(|U^-|  - |U^+|) - \frac{1}{2}\eps n 
\geq \frac{1}{2}(|U^-|  - |U^+|) - \frac{1}{4}\lambda n.
\end{align*}
Thus for each $x \in X$, we can greedily pick disjoint $Z_x' \subseteq N_D^+(x) \cap U^-$ with $|Z_x'| \leq \disc_D(x)$ and $|\cup_{x \in X} Z_x'| = \frac{1}{2}(|U^-| - |U^+|) - \frac{1}{4}\lambda n$.
Let $Y$ be the $\frac{1}{4} \lambda n$ vertices of lowest excess (i.e.\ of highest negative excess) in~$Z' := \cup_{x \in X} Z_x'$.
Set $Z_x := Z_x' \setminus Y$ and $Z := Z' \setminus Y$. 
Hence $Z := \cup_{x \in X} Z_x$ and $|Z| = \frac{1}{2}||U^+| - |U^-|| - \frac{1}{2}\lambda n$.
Also
\begin{align*}
\sum_{z \in Z} \disc^-(z) 
& = \sum_{z' \in Z'} \disc^-(z')  - \sum_{y \in Y} \disc^-(y)\\
&\le (1- \lambda/4)\sum_{z' \in Z'} \disc^-(z')
\leq \left( 1 - \lambda/4 \right)\disc(D).
\end{align*}
\end{proof}

We will construct the final graph $D'$ such that $V(D') = V(D) \setminus X$, where $U^+(D')= (U^+ \setminus X) \cup Z$ and $U^-(D') = U^- \setminus Z$, and where $\disc_{D'}(z) = 1$ for all $z \in Z$.

For each $z \in Z$, we write $x_z$ for the vertex $x$ such that $z \in Z_x$. 
Note that $x_z z \in E(D)$. 

\begin{claim}  \label{clm:DB2}
There exists a partial decomposition $\mathcal{P}_Z := \{ x_z Q_z : z \in Z\}$ of~$D$ such that each $Q_z$ is a non-empty path in $D - X$ starting at $z$ and ending in $U^- \setminus  Z$. 
Moreover, $\ex_{D - \cup \mathcal{P}_Z} (v) \ne 0$ for all $ v \in V(D) \setminus X$ and $\Delta(\cup \mathcal{P}_Z) \le \eps_1 n$.
\end{claim}

\begin{proof}[Proof of claim]
For each $z \in Z$, pick a vertex $b_z \in U^- \setminus Z$ such that 
no $v \in U^-$ is chosen more than $\disc_D(v) - 1$ times (which is possible as $|Z| \le n \le \lambda \disc(D) / 4 \le \disc(D) - \sum_{z \in Z} \disc_D^- (z)$ by \ref{itm:DBa3}) and let $e_z = z b_z$.
Define a digraph~$H$ on~$V(D) \setminus X$ with edge set $\{e_z \colon z \in Z\}$.
Note that $\Delta(H) \leq \alpha  n  \le 2 \alpha |D-X|$.
We apply Corollary~\ref{cor:manyshortpaths3} with $D - X, H, 2 \alpha$ playing the roles of $D, H, \gamma$ and obtain a set of edge-disjoint paths $\mathcal{Q} : = \{Q_z \colon z \in Z\}$ such that each $Q_z$ starts at~$z$ and ends at~$b_z$ and $\Delta(\cup \mathcal{Q}) \le \eps_1 n/2 $.
Our claim follows by our choice of $\mathcal{Q}$. 
\end{proof}

Let $D_1 := D - \cup \mathcal{P}_Z$ and write $\mathcal{Q}_Z : = \{Q_z \colon z \in Z\}$.

\begin{claim} \label{clm:DB3}
There exists a partial decomposition $\mathcal{P}_1$ of~$D_1$ such that $\cup \mathcal{P}_1 \subseteq D_1 -X$, $\disc(\cup \mathcal{P}_1) \le n$, $\Delta(\cup \mathcal{P}_1) \le \eps_1 n$ and $\disc_{\cup \mathcal{P}_1} (v) \ne 0$ if $v \notin X \cup Z$ and $\disc_{\cup \mathcal{P}_1} (v) = 0$ otherwise.
\end{claim}

\begin{proof}[Proof of claim]
Let $H$ be any digraph on~$V(D) \setminus (X \cup Z)$ with edges from $U^+$ to $U^-$ such that $1 \le d_H(v) \le |\disc_{D_1}(v)|$ for all $ v \in V(D) \setminus (X \cup Z)$. 
Note that $\Delta(H) \leq \alpha  n \le 2 \alpha |D-X|$.
By deleting edges of $H$ if necessary, we may assume that $H$ has at most $n$ edges. 
We apply Corollary~\ref{cor:manyshortpaths3} with $D_1 - X, H, 2 \alpha$ playing the roles of $D, H, \gamma$ and obtain the desired partial decomposition~$\mathcal{P}_1$. 
\end{proof}

Let $D_2 := D_1 - \cup \mathcal{P}_1$.
Note that $\delta(D_2) \ge (1- 3 \eps_1) n $. 

\begin{claim} \label{clm:DB4}
There exists a partial decomposition $\mathcal{P}_2$ of~$D_2$ such that,
writing $D_3 := D_2 - \cup \mathcal{P}_2$,
we have $d_{D_3}(x) = 0$ for all $x \in X$, $d_{D_3}(v) \ge (1- \eps'/2)n$ for all $v \notin X$, $\disc_{D_3}(z) = 0$ for all $z \in Z$,  and $\disc(D_3) \ge \lambda C n /32$
\end{claim}

\begin{proof}[Proof of Claim~\ref{clm:DB4}]
Let 
\begin{align*}
m := \max \left\{  \sum_{x \in X}\disc_D(x) , \sum_{z \in Z}\disc_D^-(z) \right\}. 
\end{align*}
Recall that $\delta(D_2) \ge (1- 3 \eps_1) n$ and 
\begin{align}
	\disc(D_2) \ge \disc(D) - 2n \ge (1- \lambda /8)\disc(D) 
	\overset{\text{\ref{itm:DBa3}}}{\ge}
	\lambda \disc(D)/8+ m  \ge \lambda C n /8 +m. \nonumber
\end{align}
Let $H$ be a digraph on~$V(D)$ with $m$ edges from $U^+$ to $U^-$ such that $d_H(v) = |\disc_{D_2}(v)|$ for all $ v \in X \cup Z$ and $d_H(v) \le |\disc_{D_2}(v)|$ otherwise.
(Such an $H$  exists by the calculation above.)
Note that $\Delta(H) \leq \alpha  n$.
We apply Corollary~\ref{cor:manyshortpaths3} with $D_2, H, \alpha$ playing the roles of $D, H, \gamma$ and obtain a partial decomposition~$\mathcal{P}_2'$ of~$D_2$ such that, writing $D_2' := D_2 - \cup \mathcal{P}_2'$, we have 
\begin{align*}
	\disc_{D_2'}(v) = 0 \text{ for all $v \in X \cup Z$}, \text{ }
	\disc(D_2') \ge \lambda C n /8   \text{ and }
	\Delta(\cup \mathcal{P}_2') \le \eps_1n. 
\end{align*}

We now apply Lemma~\ref{lma:W} with $(D_2', X, \alpha, \beta, \eps_1, 4\eps_1, \lambda C/8)$ playing the roles of $(D,W,\alpha, \beta, \gamma, \eps, C)$.
We obtain a partial decomposition $\mathcal{P}_2''$ of~$D_2'$ such that, writing $D_3 := D_2' - \cup \mathcal{P}_2''$, we have
\begin{align*}
	d_{D_3}(x)  = 0 \text{ for all $x \in X$}, \:\: 
	\disc(D_3) \ge \lambda C n/32 \:\:  \text{ and } \:\:
	\Delta(\cup \mathcal{P}_2'' - X) \le 5\eps_1 n
\end{align*}
The claim holds by setting $\mathcal{P}_2 := \mathcal{P}_2' \cup \mathcal{P}_2''$. 
\end{proof}

Finally, we show how to prove the theorem using Claim~\ref{clm:DB4}. 
Note that $\mathcal{P}_2$ is a partial decomposition of~$D$ by Proposition~\ref{pr:PartDecomp1}(b). 
Let 
\begin{align*}
D'' &: = D - \cup \mathcal{P}_2 = \cup (\mathcal{P}_Z \cup \mathcal{P}_1) \cup D_3,\\
D' & := D''-X = \cup (\mathcal{Q}_Z \cup \mathcal{P}_1) \cup (D_3-X).
\end{align*}
Since vertices of $X$ are isolated in $D_3$,
 we have $E (D'' - D') = \{ x_z z \colon z \in Z\}$. 
Therefore, by Proposition~\ref{pr:PartDecomp2}(b), (with $(D'', X, \emptyset, V(D) \setminus X)$ playing the roles of $(D,A^+,A^-,R)$) we see that if $D'$ has a perfect decomposition then so does $D''$ and hence so does $D$, a contradiction; hence $D'$ does not have a perfect decomposition, proving \ref{itm:DB1}.
Note that \ref{itm:DB2} follows from \ref{itm:DBa1}.
Since $E(D_3-X) = E(D_3)$, \ref{itm:DB3} holds by Claim~\ref{clm:DB4}. 
For all $z \in Z$, we have $\disc_{\mathcal{P}_1} (z) = 0$ and $\disc_{D_3-X} (z) = 0 $, and so $\disc_{D'}(z) = \disc_{\cup \mathcal{Q}_Z}(z) = 1$ by Claim~\ref{clm:DB2}.
Since $\mathcal{P}_2$ is a partial decomposition of~$D$, $\disc_{D''}^{\pm} (u) \le \alpha n $ for all $u \in U^{\pm}$. 
Moreover, for $u \in U^{\pm} \setminus (X \cup Z)$, $\disc_{D'}^{\pm} (u) = \disc_{D''}^{\pm} (u) \ge \disc_{\cup \mathcal{P}_1}^{\pm} (u) \ge 1$. 
Hence \ref{itm:DB4} holds.
Furthermore, we have $U^+(D') = (U^+ \setminus X) \cup Z$ and $U^-(D') = U^- \setminus  Z$.
Thus \ref{itm:DB6} holds by~\ref{itm:DBa4}.
Note that $\mathcal{Q}_Z$ and $\mathcal{P}_1$ are partial decompositions\footnote{To see this note that $\mathcal{P}_{Z}$ and $\mathcal{P}_1$ are partial decompositions of $D''$. We obtain respectively $D'$, $\mathcal{Q}_Z$, $\mathcal{P}_1$ by deleting $X$ from $D''$, $\mathcal{P}_Z$, $\mathcal{P}_1$. Then noting that $\disc_{D''}(z) = \disc_{\cup \mathcal{P}_Z \cup \mathcal{P}_1}(z)=0$ for all $z \in Z$ and that  the only edges incident to $X$ in $D''$ are the initial edges of paths in $\mathcal{P}_Z$, we can conclude $\mathcal{Q}_Z$ and $\mathcal{P}_1$ are partial decomposition of $D'$.} of $D'$, so $\disc(D') \ge \disc(D_3-X) = \disc(D_3) \ge \lambda C n /32$ implying~\ref{itm:DB5}.
\end{proof}

We now show that the digraph produced by Theorem~\ref{thm:DiscBalance} has a perfect decomposition.
Together with Theorem~\ref{thm:RemHighDisc} and Theorem~\ref{thm:DiscBalance}, this will give us all the ingredients to prove Theorem~\ref{thm:LPST2}.

\begin{theorem}
\label{thm:FinalDecomp}
Let $1/n \ll \alpha, \lambda, \eps \ll 1$. 
Suppose that $D$ is an $n$-vertex oriented graph, where 
\begin{itemize}
\item $\disc(D) \geq 2n$;  
\item $\delta(D) \geq (1 - \eps) n$; 
\item $1 \leq |\disc_{D}(v)| \leq \alpha n$ for all $v\in V(D)$;
\item $\left| |U^-(D')| - |U^+(D')| \right| \leq 2\lambda |D'|$.
\end{itemize}
Then $D$ has a perfect decomposition.
\end{theorem}

\begin{proof}
Fix a parameter $\eps'$ satisfying $1/n \ll \alpha, \lambda, \eps \ll \eps' \ll 1$ such that $\eps' n$ is an integer.
Let 
\begin{align*}
d = (1/2 - 52\eps')n \text{ and } t := n - 2d = 104 \eps'n.
\end{align*}
Arbitrarily partition $V(D)$ into $X^+, X^-, X^0$ such that 
\begin{align*}
	|X^+| &= |X^-| = d, &
	|X^0| &= t, &
	X^{\pm} \subseteq U^{\pm}:=U^{\pm}(D).
\end{align*}
(Note that such partition exists as $|U^{\pm}| \ge d$.)
Our goal is to remove a partial decomposition~$\mathcal{P}$ of~$D$ such that the resulting digraph $D' := D - \cup \mathcal{P}$ satisfies 
\begin{align}
\label{eq:degree}
\disc_{D'}(v) = 
\begin{cases}
1 &\text{if } v \in X^+; \\
0 &\text{if } v \in X^0; \\
-1 &\text{if } v \in X^-;
\end{cases}
\:\:\:\text{ and }\:\:\:
d_{D'}(v) = 
\begin{cases}
2d-1 &\text{if } v \in X^+; \\
2d &\text{if } v \in X^0; \\
2d-1 &\text{if } v \in X^-.
\end{cases}
\end{align}
Then $D'$ has a perfect decomposition~$\mathcal{P'}$ by Theorem~\ref{thm:KellyPaths} and so $\mathcal{P} \cup \mathcal{P}'$ is a perfect decomposition of $D$ (by Proposition~\ref{pr:PartDecomp1}(b)).
Thus it remains to find such a $\mathcal{P}$.

We will construct $\mathcal{P}$ as a union of three partial decompositions $\mathcal{P}_1,\mathcal{P}_2,\mathcal{P}_3$.
Let $D_0:= D$ and write $D_i := D_{i-1} - \cup \mathcal{P}_i$ for $i=1,2,3$. 
First, we reserve two multisets $A_2$ and $A_3$, which will be sets of starting and ending points of $\mathcal{P}_2$ and $\mathcal{P}_3$, respectively. 
Second, we find a partial decomposition~$\mathcal{P}_1$ such that $\disc_{D_1}(v)$ has the correct value provided $v \notin A_2 \cup A_3$ (see Claim~\ref{clm:F1}).
The partial decomposition $\mathcal{P}_2$ will ensure that $d_{D_2}(v) = 2d'-I_{X^+ \cup X^-}(v)$ for some $d' >d$.
Finally, we adjust~$d'$ to $d$ using $\mathcal{P}_3$.

Since $ \disc(D) \ge 2n$ and $ |\disc_{D}(v)| \leq \alpha n$, we know we can find vertices $x_1, \dots, x_{26\eps'n}, x'_1, \dots, x'_{26\eps'n} \in U^+$ such that $x_i \ne x'_i$ and no vertex $v \in U^+$ is chosen more than $(\disc_D(v) -1)/2$ times. 
Similarly, we are able to pick vertices $y_1, \dots, y_{26\eps'n}, y'_1, \dots, y'_{26\eps'n} \in U^-$ such that $y_i \ne y'_i$ and no vertex $v \in U^-$ is chosen more than $(|\disc_D(v)| -1)/2$ times. 
Clearly, $x_i,x'_i, y_i,y_i'$ are distinct for all~$i$. 
Let 
\begin{align*}
A_2 &:=\{x_i,x'_i, y_i,y_i' \colon i \in [25 \eps'n] \}, \\
A_3 &:=\{x_i,x'_i, y_i,y_i' \colon i \in [25\eps'n+1, 26 \eps'n] \}.
\end{align*}
For $j \in \{2,3\}$, let $\phi^+_j(v)$ (and $\phi^-_j(v)$) be the number of times that $v$ is chosen as $x_i$ or $x'_i$ (and $y_i$ or $y'_i$) in~$A_j$.
Let $\phi_j(v) := \phi^+_j(v)- \phi^-_j(v)$.
Note that $\sum_{v \in V(D)}\phi_j(v) = 0$ and $2|\phi_2(v)+\phi_3(v)| < |\disc_D(v)|$.

\begin{claim}  \label{clm:F1}
There exists a partial decomposition $\mathcal{P}_1$ of~$D$ such that, writing $D_1:=D- \cup \mathcal{P}_1$, we have  $\delta(D_1) \ge (1- \eps') n$ and for all $v \in V(D)$,
\begin{align*}
\disc_{D_1}(v) = 
\begin{cases}
2\phi_2(v)+2\phi_3(v) + \pm1  &\text{if } v \in X^{\pm}; \\
2\phi_2(v)+2\phi_3(v)  &\text{if } v \in X^0.
\end{cases}
\end{align*}
\end{claim}

\begin{proof}[Proof of claim]
Let $f:V(D) \rightarrow [n]$ be such that 
\begin{align*}
f(v) = 
\begin{cases}
\disc(v) - 2\phi_2(v)- 2\phi_3(v)  \mp 1 &\text{if } v \in X^{\pm}; \\
\disc(v) - 2\phi_2(v)- 2\phi_3(v)   &\text{if } v \in X^0.
\end{cases}
\end{align*}
Note that $\sum_{v \in V(D)} f(v) = 0$ and $|f(v)| \le \alpha n $ for all $v \in V(D)$. 
Define a directed multigraph~$H$ on~$V(D) $ such that $d^+_H(v) = \max\{f(v), 0\}$ and $d^-_H(v) = \max\{-f(v), 0\}$.
Note that $\Delta(H) \leq \alpha  n $.
We apply Corollary~\ref{cor:manyshortpaths3} with $D, H, \alpha$ playing the roles of $D, H, \gamma$ and obtain the desired partial decomposition~$\mathcal{P}_1$. 
\end{proof}

Let 
\begin{align*}
	s: = \max_{v \in V} \{ d_{D_1}(v) + I_{X^+ \cup X^-}(v)\}.
\end{align*}
Note that $d_{D_1}(v)$ is even if $v \in X^0$ and odd otherwise. 
So $s$ is even and $(1- \eps')n \le s \le n$. 
Let $d' := s/2 - 50 \eps' n$, so 
\begin{align}
\label{eqn:d'}
	d+\eps' n = (1/2 - 51 \eps') n \le  d' \le (1/2 - 50 \eps') n = d+2\eps' n.
\end{align}

\begin{claim}  \label{clm:F2}
There exists a partial decomposition $\mathcal{P}_2$ of~$D_1$ such that, for all $v \in V(D)$, $\disc_{\cup \mathcal{P}_2}(v) = 2\phi_2(v)$ and $d_{D_2}(v) = 2d'-I_{X^+ \cup X^-}(v)$,
where we write $D_2:=D_1- \cup \mathcal{P}_2$.
\end{claim}

\begin{proof}[Proof of claim]
Define $f:V(D) \rightarrow [\eps' n]$ to be such that 
\begin{align*}
f(v) = \eps' n - \frac{s - (d_{D_1}(v) + I_{X^+ \cup X^-}(v))}2.
\end{align*}
Note that $\max_{v \in V(D)} f(v) = \eps' n$.
Recall that $A_2 =\{x_i,x'_i, y_i,y_i' \colon i \in [25 \eps'n] \}$.
Write $(x^*_i, x^*_{25\eps'n+i}, y^*_i, y^*_{25\eps'n+i})  = (x_i,x'_i, y_i,y_i')$. 
 
By Proposition~\ref{pr:balance} where we take $(V(D), 25,\eps'n,x^*_i,y^*_i)$ to play the roles of $(V,t,m,x_i,y_i)$, we can find a collection of sets $T_1, \dots, T_{50 \eps' n } \subseteq V(D)$
\begin{enumerate}[label={\rm(\roman*$'$)}]
	\item for all $v \in V(D)$, $\sum_{i \in [50 \eps' n ]} I_{T_{i}}(v) =  f(v) + 49 \eps' m$;
	\label{itm:setT1}
	\item $|T_i| \ge  24 n/25$ for all $i \in [50 \eps' n ]$; \label{itm:setT2}
	\item $x_i^*,y_i^* \in T_i$  for all $i \in [50 \eps' n ]$. \label{itm:setT3}
\end{enumerate}

For $i \in[50 \eps' n]$, let $S_i := V(D) \setminus T_i$ and $H_i$ be the multidigraph on $V(D)$ with $E(H_i) = \{ x_i^*y_i^* , x_i^*y_i^*\}$.
Let $H = \bigcup_{i \in [50 \eps' n] }H_i$.
Note that $|E(H_i)| = \Delta(H_i) =2 $ and $|S_i| \le n/25 $. 
We apply Lemma~\ref{lma:manyshortpaths3} where we take $(D_1, H_i, S_i, 50 \eps' )$ to play the role of $(D, H_i, S_i, \gamma)$ and obtain edge-disjoint paths $P_1,P_1',  \ldots, P_{50 \eps' n}, P'_{50 \eps' n}$ such that both $P_i$ and $P_i'$ start at $x^*_i$ and end at $y^*_i$ and $d_{P_i \cup P_i'}(v) = 2 I_{T_{i}}(v)$ for all $ v \in V(D)$. 

Set $\mathcal{P} _2 = \{P_i,P'_i : i \in [50 \eps' n]\}$.
Note that $\disc_{\cup \mathcal{P}_2}(v) = 2\phi_2(v)$ for all $v \in V(D)$.
For all $v \in V(D)$,
\begin{align*}
d_{D_2}(v) & = d_{D_1}(v) - 2\sum_{i \in [50 \eps' n ]} I_{T_{i}}(v)
\overset{\text{\ref{itm:setT1}}}{=}
 d_{D_1}(v) - 2(f(v) + 49 \eps' n)\\
& = s - 100\eps' n - I_{X^+ \cup X^-}(v)
 = 2d' - I_{X^+ \cup X^-}(v),
\end{align*}
as required.
\end{proof}

\begin{claim}  \label{clm:F3}
There exists a partial decomposition $\mathcal{P}_3$ of~$D_2$ such that for all $v \in V(D)$, $\disc_{\cup \mathcal{P}_3}(v) = 2\phi_3(v)$ and $d_{D_3}(v) = 2d-I_{X^+ \cup X^-}(v)$,
where we write $D_3:=D_2- \cup \mathcal{P}_3$.
\end{claim}

\begin{proof}[Proof of claim]
Recall that $A_3 :=\{x_i,x'_i, y_i,y_i' \colon i \in [25\eps'n+1, 26 \eps'n] \}$. 
Let $m := d' - d-\eps' n$, so $0 \le m \le \eps'n$ by~\eqref{eqn:d'}. 

We now define multidigraphs $H_1, \dots, H_{m + \eps' n}$  on $V(D)$ as follows. Define $f(i) = 25\eps'n + i$.  
For $i \in [m]$, 
\begin{align*}
	E(H_i)&:= \{x_{f(i)}y_{f(i)},\; x_{f(i)}y_{f(i)}\}, &
	E(H_{\eps' n +i})&:= \{x_{f(i)}'y_{f(i)}',\; x_{f(i)}'y_{f(i)}'\}.
\end{align*}
For $i \in [m+1, \eps'n]$, set
\begin{align*}
		E(H_i)&:= \{x_{f(i)}y_{f(i)},\; x_{f(i)}y_{f(i)},\; x_{f(i)}'y_{f(i)}',\; x_{f(i)}'y_{f(i)}'\}.
\end{align*}
Note that $|E(H_i)| \le 4$ and $ \Delta(H_i) =2 $. 
Let $H:= \bigcup_{i \in [\eps' n + m]}H_i$.
We apply Lemma~\ref{lma:manyshortpaths3} with $(D_2, H_i, \emptyset, 2 \eps' )$ playing the roles of $(D, H_i, S_i, \gamma)$ to obtain a set of edge-disjoint paths $\mathcal{P}_3 = \{P_e \colon e \in E(H) \}$ such that $d_{\cup \mathcal{P}_3}(v) = 2 (\eps' n + m) = 2(d' - d)$ and $\disc_{\cup \mathcal{P}_3}(v) = 2\phi_3(v)$ for all $v \in V(D)$. Note that $\mathcal{P}_3$ is a partial decomposition of $D_2$ by our choice of $H_i$ and that $D_3 =D_2 - \cup \mathcal{P}_3$ satisfies the desired properties.
\end{proof}

For all $v \in V(D)$,
\begin{align*}
\disc_{D_3}(v) = \disc_{D_1}(v) -  \disc_{\cup \mathcal{P}_2}(v) - \disc_{\cup \mathcal{P}_3}(v) = 
\begin{cases}
\pm 1 &\text{if } v \in X^{\pm}; \\
0 &\text{if } v \in X^0.
\end{cases}
\end{align*}
Also $d_{D_3}(v) = 2d-I_{X^+ \cup X^-}(v)$ for all $v \in V(D)$ by Claim~\ref{clm:F3}. We are done by setting $D':=D_3$.
\end{proof}

\subsection{Final proof}

Now we can finally prove our main theorem.

\begin{proof}[Proof of Theorem~\ref{thm:LPST2}]
Assume $1/n_0 \ll 1/C \ll 1$ and that $T$ is an even tournament with $n \geq n_0$ vertices and $\disc(T) \geq Cn$.

We pick parameters $\alpha_1, \beta_1, \eps_1, \alpha_2, \beta_2, \eps_2, \lambda$ satisfying:
\begin{equation}
\label{eq:finalhierarchy}
1/n \ll 1/C \ll \beta_1 \ll \alpha_1, \beta_2 \ll \eps_1 \ll \eps_2, \lambda \ll 1 \:\:\:\text{ and }\:\:\: 1/C \ll \lambda. 
\end{equation}

By Theorem~\ref{thm:RemHighDisc} either $T$ has a perfect decomposition or there is a digraph $D_1$ satisfying the following properties:
\begin{enumerate}[label={\rm(a$_{\arabic*}$)}]
\item If $D_1$ has a perfect decomposition then $T$ has a perfect decomposition;
\item $n_1:=|D_1| \geq (1 - \beta_1)n$ with $n_1$  even;
\item $\delta(D_1) \geq (1 - \eps_1)n_1$;
\item $1 \leq |\disc_{D_1}(v)| \leq \alpha_1 n_1$ for all $v\in V(D_1)$;
\item $\disc(D_1) \geq (C/4 - 5)n_1 =:C_1n_1$.
\end{enumerate}

By Theorem~\ref{thm:DiscBalance} there exists a digraph $D_2$ satisfying the following properties:
\begin{enumerate}[label={\rm(b$_{\arabic*}$)}]
\item If $D_2$ has a perfect decomposition then $D_1$ has a perfect decomposition;
\item $n_2:=|D_2| \geq (1 - \beta_2)n_1$ with $n_2$ even; 
\item $\delta(D_2) \geq (1 - \eps_2)n_2$;
\item $1 \leq |\disc_{D_2}(v)| \leq \alpha_1 n_1 \leq 2\alpha_1 n_2$ for all $v\in V(D_2)$;
\item $\disc(D_2) \geq \lambda C_1 n_2 /32 =:C_2n_2 \geq 2n_2$;
\item $\left| |U^-(D_2)| - |U^+(D_2)| \right| \leq 2\lambda n_2$.
\end{enumerate}

Note that by \eqref{eq:finalhierarchy}, we have $1/n_2 \ll  2\alpha_1, \lambda, \eps_2 \ll 1$ since $n_2 \geq (1- \beta_1)(1 - \beta_2)n \geq n/2$.
By Theorem~\ref{thm:FinalDecomp}, $D_2$ has a perfect decomposition; hence so does $D_1$ (by~(b$_1$)) and so does $T$ (by (a$_1$)) as required. 
\end{proof}

\section{Conclusion}
\label{se:conc}

We have proved many cases of Conjecture~\ref{conj:Pull}. The obvious open problem remaining is to fill the remaining gap, that is to prove that $\pn(T) = \ex(T)$ for all even tournaments satisfying $n/2 < \disc(T) \leq Cn$ for some sufficiently large $C$. We believe that with a little work, one should be able to apply the results of K{\"u}hn and Osthus~\cite{KO} to prove the conjecture when $\ex(T)$ is very close to $n/2$ but that probably some new ideas are needed, say when $n \leq \disc(T) \leq Cn$.

Another direction, which is currently work in progress, is to investigate analogues of Conjecture~\ref{conj:Pull} for  directed graphs that are not tournaments. In forthcoming work we consider dense directed graphs as well as random and quasi-random directed graphs.


\begin{thebibliography}{10}

\bibitem{AMP} B.~Alspach, D.~Mason, and N.~Pullman,
\emph{Path numbers of tournaments}, 
J.~Combin. Theory  B {\bf 20} (1976), 222--228.

\bibitem{AP} B.~Alspach and N.~Pullman,
\emph{Path decompositions of digraphs}, 
Bull.\ Austral.\ Mat.\ Soc.\  {\bf 10} (1974), 421--427.

\bibitem{BG}
 J.~Bang-Jensen and G.~Gutin, 
 \emph{Digraphs: Theory, Algorithms and Applications}, Springer 2000.

\bibitem{CKLOT}
B.~Csaba, D.~K{\"u}hn, A.~Lo, D.~Osthus, and A.~Treglown,
\emph{Proof of the 1-factorization and Hamilton decomposition conjectures},
Mem. Amer. Math. Soc. {\bf 244} (2016).


\bibitem{Diestel}
 R.~Diestel, 
 \emph{Graph Theory}, Springer 2010.




\bibitem{HMSSY}
H.~Huang, J.~Ma, A~Shapira, B.~Sudakov, and R.~Yuster,
\emph{Large feedback arc sets, high minimum degree subgraphs, and long cycles in Eulerian digraphs},  
Combin. Probab. Comput. {\bf 22} (2013), 859--873.




\bibitem{klos} 
D.~K\"uhn, A.~Lo, D.~Osthus, and K.~Staden,
\emph{The robust component structure of dense regular graphs and applications}
 Proceedings London Mathematical Society {\bf 110} (2015), 19--56.


\bibitem{KO} D.~K\"uhn and D.~Osthus, 
\emph{Hamilton decompositions of regular expanders: a proof of Kelly's conjecture for large tournaments},
Adv. Math. {\bf 237} (2013), 62--146.



\bibitem{KO2} D.~K\"uhn and D.~Osthus,
\emph{Hamilton decompositions of regular expanders: applications}, 
J. Combin. Theory Ser. B {\bf 104} (2014), 1--27.


\bibitem{Kriv} M.~Krivelevich, 
\emph{Triangle factors in random graphs}, 
Combin. Probab. Comput.
{\bf 6} (1997), 337--347.

\bibitem{LP} A.~Lo and V.~Patel,
\emph{Hamilton Cycles in Sparse Robustly Expanding Digraphs}, Electr. J. Comb. {\bf 25}(3) (2018), P3.44.


\bibitem{Brien} R.~O'Brien,
\emph{An upper bound on the path number of digraphs}, 
J. Combin. Theory Ser. B {\bf 22} (1977), 168--174.


\bibitem{RRS}
V.~R{\"o}dl, A.~Ruci{\'n}ski, and E.~Szemer{\'e}di, 
\emph{A Dirac-type theorem for $3$-uniform hypergraphs}, 
Combin. Probab. Comput. \textbf{15}
(2006), 229--251.

 
\end{thebibliography}
\end{document}